\documentclass[12pt,a4paper]{amsart}
\usepackage{amssymb, verbatim}
\usepackage{tikz}
\usepackage{hyperref}
\theoremstyle{plain}
\newtheorem{theorem}{Theorem}[section]
\newtheorem{claim}{Claim}
\newtheorem{cor}[theorem]{Corollary}
\newtheorem{lemma}[theorem]{Lemma}
\newtheorem{prop}[theorem]{Proposition}

\theoremstyle{remark}
\newtheorem{notation}[theorem]{Notation}
\newtheorem{rmk}[theorem]{Remark}
\newtheorem*{rmk*}{Remark}

\newtheorem*{note*}{Note}
\theoremstyle{definition}
\newtheorem{defn}[theorem]{Definition}
\newtheorem*{defn*}{Definition}
\newtheorem{example}[theorem]{Example}
\newtheorem*{example*}{Example}
\newtheorem*{examples*}{Examples}

\numberwithin{equation}{section}

 \DeclareMathOperator{\Aut}{Aut} 
 \DeclareMathOperator{\id}{id} 
  \DeclareMathOperator{\Obj}{Obj}
\DeclareMathOperator{\Mor}{Mor}  
\DeclareMathOperator{\dom}{dom}  
  
 \DeclareMathOperator{\clsp}{{\overline{span}}}

\newcommand{\sk}[1]{\overline{#1}}

\newcommand{\field}[1]{\mathbb{#1}}
\newcommand{\CC}{\field{C}}
\newcommand{\FF}{\field{F}}
\newcommand{\NN}{\field{N}}
\newcommand{\N}{\field{N}}

\newcommand{\TT}{\field{T}}
\newcommand{\ZZ}{\field{Z}}

\newcommand{\Bb}{{\mathcal B}}
\newcommand{\Cc}{{\mathcal C}}

\newcommand{\Gg}{{\mathcal G}}
\newcommand{\Hh}{{\mathcal H}}

\newcommand{\Zz}{{\mathcal Z}}

\newcommand{\Exinfty}[1]{E^{\text{$#1$-$\infty$}}}
\newcommand{\Ekinfty}{\Exinfty{k}}

\newcommand{\spc}{\hskip2pt}

\author[Hazlewood, Raeburn, Sims, Webster]{Robert Hazlewood, Iain Raeburn, Aidan  Sims\\
and Samuel B.G. Webster}
\title[Higher-rank graphs]{\boldmath{On some fundamental results about higher-rank graphs and their $C^*$-algebras}}

\address{Robert Hazlewood\\ School of Mathematics and Statistics\\
University of New South Wales\\ Sydney\\
NSW 2052\\
Australia}
\email{robbiehazlewood@gmail.com}

\address{Iain Raeburn\\ Department of Mathematics and Statistics\\
University of Otago\\
PO Box 56\\Dunedin 9054\\
New Zealand}
\email{iraeburn@maths.otago.ac.nz}

\address{Aidan  Sims
and Samuel B.G. Webster\\ School of Mathematics and Applied Statistics\\
University of Wollongong\\
NSW 2522\\
Australia}
\email{asims@uow.edu.au, sbgwebster@gmail.com}

\date{8 October 2011}

\begin{document}

\begin{abstract}
Results of Fowler and Sims show that every $k$-graph is completely determined by its $k$-coloured
skeleton and collection of commuting squares. Here we give an explicit description of the $k$-graph
associated to a given skeleton and collection of squares and show that two $k$-graphs are
isomorphic if and only if there is an isomorphism of their skeletons which preserves commuting
squares. We use this to prove directly that each $k$-graph $\Lambda$ is isomorphic to the quotient
of the path category of its skeleton by the equivalence relation determined by the commuting
squares, and show that this extends to a homeomorphism of infinite-path spaces when the $k$-graph
is row finite with no sources. We conclude with a short direct proof of the characterisation,
originally due to Robertson and Sims, of simplicity of the $C^*$-algebra of a row-finite $k$-graph
with no sources.
\end{abstract}

\maketitle

\section{Introduction}

A $k$-graph is a combinatorial object akin to a directed graph, in which each path $\lambda$ has a
$k$-dimensional shape $d(\lambda) \in \NN^k$, called its degree, instead of a $1$-dimensional
length. $C^*$-algebras associated to graphs and $k$-graphs have attracted significant attention
recently because they at once encompass a great many interesting examples
\cite{BaumHajacEtAl:K-th05, Drinen2000, HS2002, PRRS2006}, and are remarkably tractable
\cite{DavidsonYang:CJM09, DavidsonYang:NYJM09, DHS2003, Evans2008, HS2004, JP2002, RS2007}. Indeed,
Spielberg \cite{Spielberg2007} showed how to construct every Kirchberg algebra from combinations of
graph $C^*$-algebras and $2$-graph $C^*$-algebras. However, $k$-graphs themselves are, from a
combinatorial point of view, substantially more complicated than their $1$-dimensional
counterparts, and one of the keys to using them effectively is a good visual description.

A crucial feature of $k$-graphs is the factorisation property, which says that, given any path
$\lambda$ and any decomposition $d(\lambda) = m+n$, there is a unique factorisation $\lambda =
\mu\nu$ such that $d(\mu) = m$ and $d(\nu) = n$. In particular, writing $e_1, \dots, e_k$ for the
generators of $\NN^k$, if $ef$ is a path with $d(e) = e_i$ and $d(f) = e_j$, then $d(ef) = e_j +
e_i$ so there is a unique expression $ef = f'e'$ where $d(f') = e_j$ and $d(e') = e_i$. This is
called a \emph{square} of $\Lambda$. We can regard the list $\Cc_\Lambda$ of all such squares as
data associated with the \emph{skeleton} of $\Lambda$, which is the $k$-coloured directed graph
$E_\Lambda$ with the same vertices as $\Lambda$ and with edges $\bigcup^k_{i=1} d^{-1}(e_k)$, where
edges of different degrees are coloured with different colours.

Theorem~2.2 of \cite{FS2002} characterises exactly which coloured graphs $E$ and collections $\Cc$
of squares arise from $k$-graphs; and \cite[Theorem~2.1]{FS2002} implies that for each such pair
$(E, \Cc)$ there is a unique $k$-graph up to isomorphism whose skeleton is $E$ and whose commuting
squares are those in $\Cc$. The latter theorem is an existence result; it does not explicitly
describe the $k$-graph $\Lambda_{E,\Cc}$. It is more or less folklore (and can be dug out of the
proof of \cite[Theorem~2.1]{FS2002}) that $\Lambda_{E,\Cc}$ can be described along the lines
outlined for $k=2$ in \cite[Section~6]{KP2000}: paths in $\Lambda_{E,\Cc}$ are described as paths
in $E$ in which the colours occur in a fixed preferred order. But this is unsatisfactory because it
is difficult to recognise a path when it is written as a concatenation of sub-paths, or to decide
when one path is a sub-path of another; to do so requires tedious calculations using the collection
$\Cc$ of squares.

In Section~\ref{sec:k-graphs} we provide a concrete description of the $k$-graph $\Lambda_{E,\Cc}$.
Inspired by the construction of $2$-graphs from two-dimensional shift-spaces in \cite{PRW2009}, we
show that the paths in $\Lambda$ can be regarded as coloured-graph morphisms from a collection of
model $k$-coloured graphs into $E$. An advantage of this construction is that under this
presentation, each path explicitly encodes all of its subpaths. In Section~\ref{sec:topology} we
use this to provide an explicit proof that $\Lambda$ is the quotient of the path category
$E^*_\Lambda$ of $E_\Lambda$ by the equivalence relation $\sim$ determined by $\Cc$. We then show
that the topology on the infinite-path space of $\Lambda$ coincides with the quotient topology on
$E^\infty{/}{\sim}$.  We also present an example showing that the corresponding statement is false
for boundary paths in non-row-finite $k$-graphs. Our final section gives a direct and elementary
proof that if $\Lambda$ is a row-finite $k$-graph with no sources, then $C^*(\Lambda)$ is simple if
and only $\Lambda$ is both aperiodic and cofinal (see Section~\ref{sec:simple} for details). This
result first appeared in \cite{RS2007}, but the proof there was indirect, proceeding via reference
to the results of \cite{KP2000}, which were proved using groupoid technology. Since aperiodicity
and cofinality have been characterised in a number of different ways in the literature, we use the
presentations which are best suited to the description of $\Lambda_{E,\Cc}$ from
Section~\ref{sec:k-graphs}: specifically, the description of aperiodicity introduced
in~\cite{RS2007}, and the cofinality condition of \cite{LS2010}. The key graph-theoretic component,
Lemma~\ref{lem_there_exists_lambda}, of our proof has already found applications elsewhere: it was
precisely the statement needed to establish the Cuntz-Krieger uniqueness theorem
\cite[Theorem~4.7]{ArandaClarkEtAl:xx11} for the Kumjian-Pask algebras introduced there.

\section{Background}

A \emph{directed graph} $E=(E^0,E^1,r,s)$ consists of countable sets $E^0, E^1$ and functions
$r,s:E^1 \to E^0$. Since all the graphs in this paper are directed, we will drop the adjective. We
call elements of $E^0$ \emph{vertices}, and elements of $E^1$ \emph{edges}. For an edge $e \in
E^1$, we call $s(e)$ the \emph{source} of $e$ and $r(e)$ the \emph{range} of $e$. A \emph{path of
length n} is a sequence $\mu = \mu_1 \mu_2\dots \mu_n$ of edges such that $s(\mu_i) = r(\mu_{i+1})$
for $1 \leq i \leq n-1$. We denote by $E^n$ the set of all paths of length $n$, and define $E^* :=
\bigcup_{n \in \NN} E^n$. We extend $r$ and $s$ to $E^*$ by setting $r(\mu) = r(\mu_1)$ and $s(\mu)
= s(\mu_n)$. By an infinite path in $E$, we mean a sequence $x = \nu_1\nu_2 \dots$ where
$r(\nu_{i+1}) = s(\nu_i)$ for all $i$, and we write $r(x) = r(\nu_1)$. We write $E^\infty$ for the
set of all infinite paths, and call $W_E := E^* \cup E^\infty$ the \emph{path space} of $E$.

For $k \in \NN$, a \emph{$k$-graph} is a pair $(\Lambda, d)$ where $\Lambda$ is a countable
category and $d$ is a functor from $\Lambda$ to $\NN^k$ which satisfies the \emph{factorisation
property}: for every $\lambda \in \Mor(\Lambda)$ and $m,n \in \NN^k$ with $d(\lambda) = m + n$,
there are unique elements $\mu,\nu \in \Mor(\Lambda)$ such that $\lambda = \mu \nu$, $d(\mu) = m$
and $d(\nu) = n$ (see \cite[Definition~1.1]{KP2000}). Elements $\lambda \in \Mor(\Lambda)$ are
called \emph{paths}, and by convention we write $\lambda \in \Lambda$ to mean $\lambda \in
\Mor(\Lambda)$. The functor $d$ is called the degree map.

For $m \in \NN^k$ and $v \in \Obj(\Lambda)$, we define $\Lambda^m := \{\lambda \in \Lambda :
d(\lambda) = m \}$ and $v\Lambda^m := \{ \lambda \in \Lambda^m : r(\lambda) = v\}$. More generally,
given $\lambda \in \Lambda$ and $F, G \subseteq \Lambda$, we define $\lambda G = \{\lambda \nu :
\nu \in G, r(\nu) = s(\lambda)\}$ and $F\lambda = \{\mu\lambda : \mu \in F, s(\mu) = r(\lambda)\}$;
and then $F \lambda G = \bigcup_{\mu \in F} \mu\lambda G = \bigcup_{\nu \in G} F\lambda\nu$.

A morphism between $k$-graphs $(\Lambda_1,d_1)$ and $(\Lambda_2,d_2)$ is a functor $f: \Lambda_1
\to \Lambda_2$ which respects the degree maps. The factorisation property implies that $v \mapsto
\id_v$ is a bijection between $\Obj(\Lambda)$ and $\Lambda^0$, allowing us to identify
$\Obj(\Lambda)$ with $\Lambda^0$. In particular, we will henceforth regard $r$ and $s$ as maps from
$\Lambda$ to $\Lambda^0$.

\section{Coloured graphs and coloured-graph morphisms}
Consider the free semigroup $\FF_k$ on $k$-generators $\{c_1, \dots c_k\}$. A \emph{$k$-coloured
graph} is a graph $E$ together with a map $c : E^1 \to \{c_1, \dots, c_k\}$, which we extend to a
functor $c : E^* \to \FF_k^+$. We write $q$ for the canonical quotient map $q : \FF_k^+ \to \NN^k$
determined by $q(c_i) = e_i$ for all $i$. So each path $x \in E^*$ has both a colouring $c(x) \in
\FF_k^+$ and a shape $q(c(x)) \in \NN^k$. If there are multiple $k$-coloured graphs around, we
write $c_E$ for the colour map associated to the graph $E$. In this paper, we will draw edges of
colour $c_1$ in blue and solid, edges of colour $c_2$ in red and dashed, and edges of colour $c_3$
in green and dotted.

A \emph{graph morphism} $\psi$ from a graph $E$ to a graph $F$ consists of functions $\psi^0:E^0\to
F^0$ and $\psi^1:E^1 \to F^1$ such that $r_F(\psi^1(e)) = \psi^0(r_E(e))$ and $s_F(\psi^1(e)) =
\psi^0(s_E(e))$ for all $e \in E^1$. Given graph morphisms $\psi : E \to F$ and $\phi : F \to G$,
we write $\phi \circ \psi$ for the graph morphism from $E$ to $G$ given by $(\phi \circ \psi)^i =
\phi^i \circ \psi^i$ for $i = 0,1$. A \emph{coloured-graph morphism} is a graph morphism $\psi$
such that $c_E(e) = c_F(\psi(e))$ for every $e \in E^1$.

The following example describes the model $k$-coloured graphs which will underly the construction
used in our main theorem in Section~\ref{sec:k-graphs}. In the example, $n + v_i$ is a formal
symbol intended to suggest an edge of colour $c_i$ pointing from the integer-grid point $n + e_i$
to the integer-grid point $n$.

\begin{example}\label{eg:Esubkm}
For $m \in (\NN \cup \{ \infty \})^k$, we define a coloured graph $E_{k,m}$ by
\begin{gather*}
E_{k,m}^0 = \{ n \in \NN^k: 0 \leq n \leq m \},\qquad
    E_{k,m}^1 = \{ n+v_i: n, n+e_i \in E_{k,m}^0\},\\
r(n+v_i) = n, \qquad s(n+v_i) = n+e_i \qquad\text{and}\qquad c(n+v_i) = c_i.
\end{gather*}
For $n+v_i \in E^1$ and $m \in \NN^k$, we define $(n+v_i)+m := (n+m)+v_i$. For $x\in E^1$ with
$c(x) = c_j$, it is unambiguous and often useful to write $v_{c(x)}:=v_j$. Given a coloured-graph
morphism $\lambda : E_{k,m} \to E$ we say $\lambda$ has degree $m$ and write $d(\lambda) = m$, and
define $r(\lambda):= \lambda(0)$ and $s(\lambda) := \lambda(m)$.
\end{example}

Given a $k$-coloured graph $E$ and distinct $i,j \in \{1,\dots,k\}$, an \emph{$\{i,j\}$-square} (or
just a \emph{square}) in $E$ is a coloured-graph morphism $\phi:E_{k,e_i+e_j} \to E$. If
$\lambda:E_{k,m} \to E$ is a coloured-graph morphism and $\phi$ is a square in $E$, then $\phi$
\emph{occurs in} $\lambda$ if there exists $n \in \NN^k$ such that $\phi(x) = \lambda(x + n)$ for
all $x \in E_{k,e_i+e_j}$.

Let $E$ be a $k$-coloured graph. A \emph{complete collection of squares} is a collection $\Cc$ of
squares in $E$ such that for each $x \in E^*$ with $c(x) = c_ic_j$ and $i \not= j$, there exists a
unique $\phi \in \Cc$ such that $x = \phi(v_i)\phi(e_i+v_j)$. We write $\phi(v_i)\phi(e_i+v_j) \sim_\Cc
\phi(v_j)\phi(e_j+v_i)$, so for each $c_ic_j$-coloured path $x \in E^*$, there is a unique
$c_jc_i$-coloured path $y$ such that $x \sim_\Cc y$. If $\Cc$ is clear from context, we just write
$x \sim y$. A coloured-graph morphism $\lambda : E_{k,m} \to E$ is \emph{$\Cc$-compatible} if every
square occurring in $\lambda$ belongs to $\Cc$.

For $p,q \in \NN^k$ with $p \leq q$, define $E_{k,[p,q]}$ to be the subgraph of $E_{k,q}$ such that
\begin{align*}
E_{k,[p,q]}^0 &= \{ n \in \NN^k : p \leq n \leq q\},\\
E_{k,[p,q]}^1 &= \{ x\in E_{k,q}^1 : s(x),r(x) \in E_{k,[p,q]}^0 \}.
\end{align*}

Given a coloured-graph morphism $\lambda:E_{k,m} \to E$ and $p,q \in \NN^k$ such that $p \leq q
\leq m$, define $\lambda|_{E_{k,[p,q]}}^* : E_{k, q-p} \to E$ by
\begin{equation}\label{starrestriction}
    \lambda|_{E_{k,[p,q]}}^*(a) =  \lambda(p+a).
\end{equation}
The star is to remind us that this non-standard restriction involves a translation.

We say a complete collection of squares $\Cc$ in a $k$-coloured graph $E$ is \emph{associative} if
for every path $fgh$ in $E$ such that $f,g,h$ are edges of distinct colour, the edges
$f_1,f_2,g_1,g_2,h_1,h_2$ and $f^1, f^2, g^1, g^2, h^1, h^2$ determined by
\begin{equation}\label{eq:fgh rearrangement}
\parbox[c]{0.9\textwidth}{\begin{tikzpicture}[scale=1.75]
    \node at (-2.4,0.5) {\parbox{8truecm}{
        $fg \sim g^1 f^1$, $f^1h \sim h^1 f^2$, and $g^1h^1 \sim h^2 g^2$\\ \\
        $gh \sim h_1 g_1$, $f h_1 \sim h_2 f_1$, and $f_1 g_1 \sim g_2 f_2$ \\}};
    \node[inner sep=1.5pt, circle, fill=black] (000) at (0,0,0) {};
    \node[inner sep=1.5pt, circle, fill=black] (001) at (0,0,-1) {};
    \node[inner sep=1.5pt, circle, fill=black] (010) at (0,1,0) {};
    \node[inner sep=1.5pt, circle, fill=black] (011) at (0,1,-1) {};
    \node[inner sep=1.5pt, circle, fill=black] (100) at (1,0,0) {};
    \node[inner sep=1.5pt, circle, fill=black] (110) at (1,1,0) {};
    \node[inner sep=1.5pt, circle, fill=black] (111) at (1,1,-1) {};
    \draw[thick,red,-latex,dashed] (010)--(000) node[pos=0.55,anchor=east,inner sep=1pt] {\color{black}\tiny$g^1$};
    \draw[thick,red,-latex,dashed] (110)--(100) node[pos=0.55,anchor=east,inner sep=1pt] {\color{black}\tiny$g$};
    \draw[thick,red,-latex,dashed] (011)--(001) node[pos=0.55,anchor=east,inner sep=1pt] {\color{black}\tiny$g^2$};
    \draw[thick,green!50!black,-latex,dotted] (001)--(000) node[pos=0.4,anchor=south east,inner sep=0pt] {\color{black}\tiny$h^2$};
    \draw[thick,green!50!black,-latex,dotted] (011)--(010) node[pos=0.5,anchor=south east,inner sep=0pt] {\color{black}\tiny$h^1$};
    \draw[thick,green!50!black,-latex,dotted] (111)--(110) node[pos=0.5,anchor=south east,inner sep=0pt] {\color{black}\tiny$h$};
    \draw[thick,blue,-latex,solid] (100)--(000) node[pos=0.4,anchor=south,inner sep=1pt] {\color{black}\tiny$f$};
    \draw[thick,blue,-latex,solid] (110)--(010) node[pos=0.4,anchor=south,inner sep=1pt] {\color{black}\tiny$f^1$};
    \draw[thick,blue,-latex,solid] (111)--(011) node[pos=0.4,anchor=south,inner sep=1pt] {\color{black}\tiny$f^2$};
\begin{scope}[xshift=1.6cm]
    \node[inner sep=1.5pt, circle, fill=black] (000) at (0,0,0) {};
    \node[inner sep=1.5pt, circle, fill=black] (001) at (0,0,-1) {};
    \node[inner sep=1.5pt, circle, fill=black] (011) at (0,1,-1) {};
    \node[inner sep=1.5pt, circle, fill=black] (100) at (1,0,0) {};
    \node[inner sep=1.5pt, circle, fill=black] (101) at (1,0,-1) {};
    \node[inner sep=1.5pt, circle, fill=black] (110) at (1,1,0) {};
    \node[inner sep=1.5pt, circle, fill=black] (111) at (1,1,-1) {};
    \draw[thick,blue,-latex,solid] (100)--(000) node[pos=0.6,anchor=north,inner sep=1pt] {\color{black}\tiny$f$};
    \draw[thick,blue,-latex,solid] (101)--(001) node[pos=0.6,anchor=north,inner sep=1pt] {\color{black}\tiny$f_1$};
    \draw[thick,blue,-latex,solid] (111)--(011) node[pos=0.6,anchor=north,inner sep=1pt] {\color{black}\tiny$f_2$};
    \draw[thick,red,-latex,dashed] (111)--(101) node[pos=0.4,anchor=west,inner sep=1pt] {\color{black}\tiny$g_1$};
    \draw[thick,red,-latex,dashed] (110)--(100) node[pos=0.4,anchor=west,inner sep=1pt] {\color{black}\tiny$g$};
    \draw[thick,red,-latex,dashed] (011)--(001) node[pos=0.4,anchor=west,inner sep=1pt] {\color{black}\tiny$g_2$};
    \draw[thick,green!50!black,-latex,dotted] (001)--(000) node[pos=0.4,anchor=north west,inner sep=0pt] {\color{black}\tiny$h_2$};
    \draw[thick,green!50!black,-latex,dotted] (101)--(100) node[pos=0.4,anchor=north west,inner sep=0pt] {\color{black}\tiny$h_1$};
    \draw[thick,green!50!black,-latex,dotted] (111)--(110) node[pos=0.5,anchor=north west,inner sep=0pt] {\color{black}\tiny$h$};
\end{scope}
\end{tikzpicture}}
\end{equation}
satisfy $f^2 = f_2, g^2 = g_2$ and $h^2 = h_2$.

Let $E$ be a $k$-coloured graph, and $m \in \NN^k\setminus\{0\}$. Fix $x\in E^*$ and a
coloured-graph morphism $\lambda : E_{k,m} \to E$. We say $x$ \emph{traverses} $\lambda$ if the
shape $q(c(x))$ of $x$ is equal to $d(\lambda)$ and $\lambda(q(c(x_1\dots x_{l-1})) + v_{c(x_l)}) =
x_l$ for all $0 < l \leq |m|$. If $m=0$, then $x \in E^0$ and $\dom(\lambda) = \{0\}$, and we say
$x$ traverses $\lambda$ if $x = \lambda(0)$. Observe that for any coloured-graph morphism
$\lambda$, and any decomposition $d(\lambda) = e_{i_1} + e_{i_2} + \dots + e_{i_{|\lambda|}}$ there
is a corresponding path $x := \lambda(0 + v_{i_1}) \lambda(e_{i_1} + v_{i_2}) \dots \lambda
\big((d(\lambda) - e_{i_{|\lambda|}}) + v_{i_{|\lambda|}}\big)$ which traverses $\lambda$; in
particular, for every finite coloured-graph morphism $\lambda$ there is a path which traverses
$\lambda$.

We can also make sense of infinite coloured paths which traverse infinite coloured-graph morphisms.
If $x \in E^\infty$ and $\lambda : E_{k,p} \to E$ is a coloured-graph morphism of non-finite degree
(so $p \in (\NN \cup \{\infty\})^k \setminus \NN^k$), then we say that $x$ traverses $\lambda$ if
$x_1 \dots x_n$ traverses $\lambda|_{E_{k,d(x_1 \dots x_n)}}$ for every $n \in \NN$.

\begin{rmk}\label{rmk:composite traversals}
Let $E$ be a $k$-coloured graph and let $\lambda : E_{k,m} \to E$ be a coloured-graph morphism
where $m \in \NN^k$. Fix $p \le m$. If $x \in E^*$ traverses $\lambda|_{[0,p]}$ and $y \in E^*$
traverses of $\lambda|^*_{[p,m]}$, then $d(\lambda) = m = p + (m-p) = q(c(x)) + q(c(y))$, and for
$l \le |xy| = |x|+|y|$, we have
\begin{align*}
    \lambda(d((xy)_1\dots (xy)_{l-1}) + v_{c((xy)_l)})
        &=\begin{cases}
            \lambda|_{[0,p]}(q(c(x_1\dots x_{l-1})) + v_{c(x)_l}) & \text{ if $l \le |x|$}\\
            \lambda|^*_{[p,m]}(q(c(y_1\dots y_{l-p-1})) + v_{c(y_{l-p})}) & \text{ otherwise}
        \end{cases} \\
        &=\begin{cases}
             x_l  & \text{ if $l \le |x|$}\\
             y_{l-p} &\text{ otherwise,}
        \end{cases}
\end{align*}
so $xy$ traverses $\lambda$.
\end{rmk}

\section{From \texorpdfstring{$k$}{k}-coloured graphs to \texorpdfstring{$k$}{k}-graphs.}\label{sec:k-graphs}

In this section we present an explicit description of the unique $k$-graph associated to a $k$-coloured graph $E$ and complete collection $\Cc$ of squares in $E$ which is associative (see Theorem~\ref{colour to
rank}).

We begin by showing how a $k$-graph defines a skeleton and a collection of squares. If $\Lambda$ is
a $k$-graph, $\lambda \in \Lambda$, and $m \le n \le d(\lambda)$, then we write $\lambda(m,n)$ for
the unique element of $\Lambda^{m-n}$ such that $\lambda = \lambda'\lambda(m,n)\lambda''$ with
$d(\lambda') = m$ and $d(\lambda'') = d(\lambda) - n$. We write $\lambda(n)$ for $s(\lambda(0,n))
\in \Lambda^0$.

\begin{defn}\label{dfn:E_Lambda def}
Let $\Lambda$ be a $k$-graph. We define a coloured graph $E_\Lambda$ and a collection $\Cc_\Lambda$
of squares associated to $\Lambda$  as follows. Let $E_\Lambda$ be the $k$-coloured graph with
$E_\Lambda^0 = \{\sk{v} : v \in \Lambda^0\}$, $E_\Lambda^1 = \bigcup_{i = 1}^k \{\sk{f} : f \in
\Lambda^{e_i}\}$, and $c(\sk{f}) = c_i \iff d(f)=e_i$. Define $\pi : E^0_\Lambda \to \Lambda$ by
$\pi(\sk{v}) = v$ and $\pi : E_\Lambda^1 \to \Lambda$ by $\pi(\sk{f}) = f$, and extend this to a
map $\pi : E^*_\Lambda \to \Lambda$ by $\pi(\sk{f_1} \dots \sk{f_n}) = f_1 \cdots f_n$. For
distinct $i,j \le k$ and $\lambda \in \Lambda^{e_i + e_j}$ define a coloured-graph morphism
$\phi_\lambda : E_{k, e_i + e_j} \to E_\Lambda$ by
\begin{equation}\label{eq:phi_lambda}
\phi_\lambda^0(n) = \sk{\lambda(n)}\quad\text{ and }\quad \phi_\lambda^{1}(n + v_i) := \sk{\lambda(n, n+e_i)}.
\end{equation}
Let $\Cc_\Lambda := \bigcup_{i < j \le k} \{\phi_\lambda : \lambda \in \Lambda^{e_i + e_j}\}$. We
call $E_\Lambda$ the \emph{skeleton} of $\Lambda$.
\end{defn}

\begin{lemma}\label{lem:k-gr -> skeleton}
Let $\Lambda$ be a $k$-graph. Fix distinct $i,j \le k$ and $\lambda \in \Lambda^{e_i + e_j}$. Then
$\phi_\lambda$ is the unique coloured-graph morphism from $E_{k, e_i + e_j} \to E_\Lambda$ such
that
\begin{equation}\label{eq:pi-phi interaction}
\pi(\phi_\lambda(0+v_i)\phi_\lambda(e_i + v_j)) = \lambda =
\pi(\phi_\lambda(0+v_j)\phi_\lambda(e_j + v_i)).
\end{equation}
Moreover $\Cc_\Lambda$ is a complete collection of squares in $E_\Lambda$ which is associative.
\end{lemma}
\begin{proof}
Fix distinct $i, j \le k$, and $\lambda \in \Lambda^{e_i + e_j}$. Then
\[
    \pi(\phi_\lambda(0+v_i)\phi_\lambda(e_i + v_j))
        = \pi\big(\sk{\lambda(0, e_i)}\spc\sk{\lambda(e_i, e_i+e_j)}\big)
        = \lambda(0, e_i)\lambda(e_i, e_i + e_j)
        = \lambda.
\]
The symmetric calculation shows that $\pi(\phi_\lambda(0+v_j)\phi_\lambda(e_j + v_i)) = \lambda$
also. Hence $\phi_\lambda$ satisfies~\eqref{eq:pi-phi interaction}. To see that it is the unique
such coloured-graph morphism, suppose that $f \in c^{-1}(i)$ and $g \in c^{-1}(j)$ and $\pi(fg) =
\lambda$. Then the factorisation property forces $\pi(f) = \lambda(0, e_i)$ and $\pi(g) =
\lambda(e_i, e_i + e_j)$. Since $\pi$ is injective on $E^1_\Lambda$, it follows that $f =
\sk{\lambda(0, e_i)} = \phi_\lambda(0 + v_i)$ and $g = \sk{\lambda(e_i, e_i + e_j)} =
\phi_\lambda(e_i+v_j)$. A symmetric argument applies with $i$ and $j$ interchanged, and this proves
the first statement of the lemma.

To see that the collection $\Cc_\Lambda$ is complete, fix $f,g \in E_\Lambda^1$ with $s(f) = r(g)$
and $c(f)\not= c(g)$, say $c(f) = c_i$ and $c(g) = c_j$. Then $\pi(f) \in \Lambda^{e_i}$ and
$\pi(g) \in \Lambda^{e_j}$, so $\pi(fg) \in \Lambda^{e_i + e_j}$, and the factorisation property
ensures that $fg$ traverses $\phi_{\pi(fg)}$. Moreover, if $\lambda \in \Lambda^{e_i+e_j}$ is
another path such that $fg$ traverses $\phi_\lambda$, then
\[
\lambda = \lambda(0, e_i)\lambda(e_i, e_i+e_j) = \pi(\phi_\lambda(0+v_i)\phi_\lambda(e_i + v_j)) = \pi(fg),
\]
so $\phi_{\pi(fg)}$ is the unique element of $\Cc_\Lambda$ such that $fg$ traverses
$\phi_{\pi(fg)}$. For the associativity condition, suppose we have $f,g,h, f^i, g^i, h^i$, and
$f_i, g_i, h_i$ as in~\eqref{eq:fgh rearrangement}. By associativity of composition in $\Lambda$,
we have
\[
    \pi(h_2g_2f_2) = \pi(fgh) = \pi(h^2g^2f^2),
\]
so the factorisation property in $\Lambda$ forces $\pi(h_2) = \pi(h^2)$, $\pi(g_2) = \pi(g^2)$ and
$\pi(f_2) = \pi(f^2)$. Since $\pi$ is injective on $E^1_\Lambda$, it follows that $h_2 = h^2$, $g_2
= g^2$ and $f_2 = f^2$ as required.
\end{proof}

\begin{notation}\label{ntn:Lambda_E setup}
Let $E$ be a $k$-coloured graph, and let $\Cc$ be a complete collection of squares in $E$ which is
associative. For each $m \in \NN^k$, we write $\Lambda_{(E,\Cc)}^m$ for the set of all
$\Cc$-compatible coloured-graph morphisms $E_{k,m}\rightarrow E$. Let $\Lambda_{(E,\Cc)} :=
\bigcup_{m \in \NN^k} \Lambda_{(E,\Cc)}^m$. Let $d : \Lambda_{(E,\Cc)} \to \NN^k$ and $r,s :
\Lambda_{(E,\Cc)} \to \Lambda^0_E$ be as defined in Example~\ref{eg:Esubkm}. For $v \in E^0$ we
define $\lambda_v : E_{k,0} \to E$ by $\lambda_v(0) = v$, and for $1 \le i \le k$ and $f \in E^1$
with $c(f) = c_i$ we define $\lambda_f : E_{k, e_i} \to E$ by $\lambda_f(0) = r(f)$,
$\lambda_f(e_i) = s(f)$ and $\lambda_f(0+v_i) = f$.
\end{notation}

Our first main theorem shows that the notation above describes a $k$-graph whose skeleton is
isomorphic to $E$ under an isomorphism which carries the commuting squares of $\Lambda$ to the
elements of $\Cc$.

\begin{theorem}\label{colour to rank}
Fix a $k$-coloured graph $E$ and a complete collection of squares $\Cc$ in $E$ which is
associative. If $\mu:E_{k,m}\rightarrow E$ and $\nu:E_{k,n}\rightarrow E$ are $\Cc$-compatible
coloured-graph morphisms such that $s(\mu)=r(\nu)$, then there exists a unique $\Cc$-compatible
coloured-graph morphism $\mu\nu:E_{k,m+n}\rightarrow E$ such that $(\mu\nu)|_{E_{k,m}}=\mu$ and
$(\mu\nu)|_{E_{k,[m,m+n]}}^*=\nu$. Under this composition map, the set $\Lambda =
\Lambda_{(E,\Cc)}$ of Notation~\ref{ntn:Lambda_E setup}, endowed with the structure maps defined
there, is a $k$-graph. There is an isomorphism $\rho : E \to E_\Lambda$ such that $\rho^0(v) =
\sk{\lambda_v}$ for all $v \in E^0$ and $\rho^1(f) = \sk{\lambda_f}$ for all $f \in E^1$; and this
$\rho$ satisfies $\rho \circ \phi \in \Cc_\Lambda$ for all $\phi \in \Cc$.
\end{theorem}

Our second main theorem says that the $k$-graph $\Lambda_{(E,\Cc)}$ is uniquely determined, up to
isomorphism, by the isomorphism class of $(E, \Cc)$.

\begin{theorem}\label{thm:uniqueness of Lambda_E} Fix a $k$-graph $\Gamma$, a $k$-coloured graph $E$
and a complete collection $\Cc$ of squares in $E$ which is associative. Suppose that $\psi :
E_\Gamma \to E$ is a coloured-graph isomorphism such that $\psi \circ \phi \in \Cc$ for all $\phi
\in \Cc_\Gamma$. Then for each $\gamma \in \Gamma$ there is a $\Cc$-compatible coloured-graph
morphism $\theta_\gamma : E_{k, d(\gamma)} \to E$ such that
\begin{align}
    \theta_\gamma^0(m) &= \psi^0(\sk{\gamma(m)}) \quad\text{ for $m \in E_{k, d(\gamma)}^0$, and} \label{eq:theta(gamma)0}\\
    \theta_\gamma^1(m+v_i) &= \psi^1(\sk{\gamma(m, m+e_i)}) \quad\text{ for $m + v_i \in E_{k, d(\gamma)}^1$.}\label{eq:theta(gamma)1}
\end{align}
Moreover, the map $\theta : \gamma \mapsto \theta_\gamma$ is an isomorphism $\Gamma \cong
\Lambda_{(E,\Cc)}$.
\end{theorem}

The key technical result which we need to prove Theorems \ref{colour to
rank}~and~\ref{thm:uniqueness of Lambda_E} says that every path in the coloured graph $E$
determines a unique element of $\Lambda$. We first use the associativity condition to prove this in
the special case of a tri-coloured path of length three, and then deal with arbitrary paths using
an inductive argument.

\begin{lemma}\label{trav}
Let $E$ be a $k$-coloured graph and let $\Cc$ be a complete collection of squares
in $E$ which is associative. If $f,g,h \in E^1$ are of distinct colour and $fgh$ is a path in $E$, then there is a
unique $\Cc$-compatible coloured-graph morphism $\lambda:E_{k,d(fgh)} \to E$ such that $fgh$
traverses $\lambda$.
\end{lemma}
\begin{proof}
The completeness of $\Cc$ implies that there exist paths $f^i, g^i, h^i$ and $f_i, g_i, h_i$
satisfying the equations~\eqref{eq:fgh rearrangement}. Let $\lambda$ be the coloured-graph morphism
such that each of $fgh$, $fh_1g_1$, $h_2f_1g_1$, $h_2g_2f_2$, $g^1f^1h$, and $g^1h^1f^2$ traverses
$\lambda$. Associativity of $\Cc$ ensures that $\lambda$ is $\Cc$-compatible. Since the values of
the $f^i,g^i,h^i$ and $f_i,g_i,h_i$ are determined by $f$, $g$, $h$ and $\Cc$, if $fgh$ traverses
$\mu$ also, then $\mu = \lambda$.
\end{proof}

\begin{prop}\label{snake_giving_C_compatible_morphism}
Let $E$ be a $k$-coloured graph and let $\Cc$ be a complete collection of squares
in $E$ which is associative. For every $x \in E^*$ there is a unique $\Cc$-compatible coloured-graph morphism $\lambda_x
: E_{k,d(x)}\rightarrow E$ such that $x$ traverses $\lambda_x$.
\end{prop}

\begin{rmk}
The notation of Proposition~\ref{snake_giving_C_compatible_morphism} is consistent with that of
Notation~\ref{ntn:Lambda_E setup} since $\lambda_v$ and $\lambda_f$ (see
Notation~\ref{ntn:Lambda_E setup}) are the unique morphisms such that $v$ traverses $\lambda_v$ and $f$
traverses $\lambda_f$.
\end{rmk}

\begin{proof}[Proof of Proposition~\ref{snake_giving_C_compatible_morphism}]
We prove this by induction on $|x|$. If $|x| = 0$, then the result is trivial.

Now suppose as an inductive hypothesis that for every $y \in E^*$ with $|y| \le n$, the path $y$
traverses a unique coloured-graph morphism $\lambda_y : E_{k,d(y)} \to E$. Fix a path $x \in E^*$
with $|x| = n+1$, and express $x = y f$ where $f \in E^1$, with $c(f) = c_i$, say.

Let $m := q(c(y))$. By the inductive hypothesis, $y$ traverses a unique $\Cc$-compatible
coloured-graph morphism $\lambda_y$. We complete the proof by consideration of three cases: $|\{j
\not= i : m_j > 0\}| = 0$, $|\{j \not= i : m_j > 0\}| = 1$, and then $|\{j \not= i : m_j > 0\}| \ge
2$.

Suppose first that $|\{j \not= i : m_j > 0\}| = 0$. Then $E_{k,d(x)} = E_{k,m} \cup
E_{k,[m,m+e_i]}$, so the formulae
\begin{equation}\label{eq:1-dim lambda_x}
\lambda_x|_{E_{k,m}} = \lambda_y,\quad \lambda_x(m + v_i) = f\quad\text{ and }\quad \lambda_x(m+e_i) = s(f)
\end{equation}
completely specify a coloured-graph morphism $\lambda_x$ such that $x$ traverses $\lambda_x$.
Furthermore, $\lambda_x$ is the unique such coloured-graph morphism: if $x$ also traverses $\mu$,
then $\mu$ satisfies the formulae~\eqref{eq:1-dim lambda_x}. This completes the proof when $|\{j
\not= i : m_j > 0\}| = 0$.

Suppose for the rest of the proof that $|\{j \not= i : m_j > 0\}| \ge 1$ (we will consider
separately later the cases $|\{j \not= i : m_j > 0\}| = 1$ and $|\{j \not= i : m_j > 0\}| \ge 2$).
Then
\begin{equation}\label{eq:decomp}
E_{k,m+e_i}
    = E_{k,m}
        \cup \Big(\bigcup_{j\not=i, m_j > 0} E_{k,m+e_i-e_j}\Big)
        \cup \Big(\bigcup_{j \not= i, m_j > 0} E_{k, [m-e_j, m+e_i]}\Big).
\end{equation}
For each $j\not=i$ such that $m_j > 0$, fix, for the remainder of the proof, a path $z^j$ which
traverses $\lambda_y|_{E_{k, m-e_j}}$.

\begin{claim}\label{claim:building lambda^j}
Suppose that $j \not= i$ satisfies $m_j > 0$. Let $\phi^j$ be the unique square in $\Cc$ traversed by
$\lambda_y((m-e_j)+v_j)f$. Let $g^j = \phi^j(0 + v_i)$ and $h^j = \phi^j(e_i + v_j)$, so $g^jh^j \sim \lambda_y((m-e_j)+v_j)f$. Then there is a unique coloured-graph morphism $\lambda^j : E_{k, m-e_j+e_i} \to E$ such that $\lambda^j|_{E_{k, m-e_j}} = \lambda_y|_{E_{k, m-e_j}}$ and $\lambda^j\big((m-e_j)+v_i\big) = g^j$.
\end{claim}

To prove Claim~\ref{claim:building lambda^j}, observe that $|z^jg^j| = n$, so the inductive
hypothesis implies that $z^jg^j$ traverses a unique $\Cc$-compatible coloured-graph morphism
$\lambda^j$. Since $z^j$ traverses both $\lambda^j|_{E_{k, m-e_j}}$ and $\lambda_y|_{E_{k,
m-e_j}}$, the inductive hypothesis implies that the two are equal. This proves
Claim~\ref{claim:building lambda^j}.

Suppose now that $|\{j \not= i : m_j > 0\}| = 1$; let $j$ be the unique element of this set. Then
Claim~\ref{claim:building lambda^j} and~\eqref{eq:decomp} imply that there is a well-defined
function $\lambda_x : E_{k, m+e_i} \to E$ such that
\begin{equation}\label{eq:2-dim lambda_x}
    \lambda_x|_{E_{k,m}} = \lambda_y,\quad \lambda_x|_{E_{k, m+e_i-e_j}} = \lambda^j\quad\text{ and }\quad
        \lambda_x|^*_{E_{k, [m-e_j, m+e_i]}} = \phi^j.
\end{equation}
This $\lambda_x$ is $\Cc$-compatible by construction, and $x$ traverses $\lambda_x$. For
uniqueness, fix a $\Cc$-compatible coloured-graph morphism $\mu$ traversed by $x$.
Then $z^jg^jh^j$ traverses $\mu$. Hence $y$ traverses $\mu|_{E_{k,m}}$ and $z^jg^j$ traverses
$\mu|_{E_{k, m-e_j+e_i}}$. The inductive hypothesis forces $\mu|_{E_{k,m}} = \lambda_y$ and
$\mu|_{E_{k, m-e_j+e_i}} = \lambda^j$. That $\mu$ is $\Cc$-compatible implies that $\mu|^*_{E_{k,
[m-e_j, m+e_i]}} = \phi^j$. So $\mu = \lambda_x$. This proves the lemma when there is a unique $j
\not= i$ such that $m_j \not= 0$ as claimed.

We now consider the last remaining case: suppose that there are at least two distinct $j,l \not= i$
such that $m_j, m_l > 0$.

\begin{claim}\label{claim:distinct j} For distinct $j,l \not= i$ with $m_j, m_l
\not= 0$, we have $\lambda^j|_{E_{k, m+e_i-e_j-e_l}} = \lambda^l|_{E_{k, m+e_i-e_j-e_l}}$.
\end{claim}

To establish Claim~\ref{claim:distinct j}, observe that since $i,j,l$ are all different,
Lemma~\ref{trav} implies that $\lambda_y\big((m - e_j - e_l)+v_l\big) \lambda_y\big((m-e_j)+v_j) f$
traverses a unique $\Cc$-compatible graph morphism $\psi^{j,l}$. We show that
\begin{equation}\label{eq:agrees with psi}
\lambda^j|^*_{E_{k, [m-e_l-e_j, m+e_i-e_j]}} = \psi^{j,l}|_{E_{k,e_i+e_l}}
    \quad\text{and}\quad
\lambda^l|^*_{E_{k, [m-e_l-e_j, m+e_i-e_l]}} = \psi^{j,l}|_{E_{k,e_i+e_j}}.
\end{equation}
By symmetry, it suffices to establish that $\lambda^j|^*_{E_{k, [m-e_l-e_j, m+e_i-e_j]}} =
\psi^{j,l}|_{E_{k,e_i+e_l}}$. Since $\lambda_y\big((m-e_j)+v_j\big)f \sim g^j h^j$, and since
$\psi^{j,l}$ is a $\Cc$-compatible coloured-graph morphism,
\[
\lambda_y\big((m - e_j - e_l)+v_l\big)g^jh^j = \lambda^j\big((m - e_j - e_l)+v_l\big) \lambda^j\big((m-e_j)+v_i\big)h^j
    \qquad\text{traverses $\psi^{j,l}$}.
\]
Since $\Cc$ is a complete collection of squares, $\lambda^j|^*_{E_{k, [m-e_l-e_j, m+e_i-e_j]}} =
\psi^{j,l}|_{E_{k,e_i+e_l}}$. This proves~\eqref{eq:agrees with psi}.

To complete the proof of Claim~\ref{claim:distinct j}, note that
\[
\lambda^j|_{E_{k, m-e_j-e_l}} = \lambda_y|_{E_{k, m-e_j-e_l}} = \lambda^l|_{E_{k, m-e_j-e_l}}.
\]
Suppose that $z$ traverses this morphism. Equation~\eqref{eq:agrees with psi} implies that $z
\psi^{j,l}(0+v_i)$ traverses each of $\lambda^j|_{E_{k, m+e_i-e_j-e_l}}$ and $\lambda^l|_{E_{k,
m+e_i-e_j-e_l}}$. The inductive hypothesis now establishes Claim~\ref{claim:distinct j}.

For $j \not= i$ such that $m_j > 0$, let $\phi^j$ and $\lambda^j$ be as in
Claim~\ref{claim:building lambda^j}. Then Claim~\ref{claim:distinct j} implies that the formulae
\[
\lambda_x|_{E_{k,m}} = \lambda_y|_{E_{k,m}},
    \quad \lambda_x|_{E_{k,m+e_i-e_j}} = \lambda^j,\quad\text{and}\quad
\lambda_x|^*_{E_{k,[m-e_j,m+e_i]}} = \phi^j
\]
determine a well-defined coloured-graph morphism $\lambda_x : E_{k,m+e_i} \to E$. Moreover
$\lambda_x$ is $\Cc$-compatible because each square occurring in $\lambda_x$ occurs in $\lambda_y$,
in one of the $\lambda^j$ or in one of the $\phi^j$.

To see that $\lambda_x$ is the unique $\Cc$-compatible coloured-graph morphism which $x$ traverses,
fix a $\Cc$-compatible coloured-graph morphism $\mu$ traversed by $x$. Then $y$ traverses
$\mu|_{E_{k,m}}$, so the inductive hypothesis implies that $\mu|_{E_{k,m}} = \lambda_y$. Fix $j
\not= i$ such that $m_j > 0$. That $\Cc$ is a complete collection of squares and that
$\lambda_y\big((m-e_j)+v_j\big)f$ traverses $\mu|^*_{E_{k,[m-e_j, m+e_i]}}$ implies that
$\mu|^*_{E_{k,[m-e_j, m+e_i]}} = \phi^j$. In particular, $\mu\big((m-e_j)+v_i\big) = g^j$, and
hence $z^jg^j$ traverses $\mu|_{E_{k, m-e_j+e_i}}$. The the inductive hypothesis forces
$\mu|_{E_{k,m+e_i-e_j}} = \lambda^j$. It now follows from~\eqref{eq:decomp} that $\mu = \lambda_x$.
\end{proof}

\begin{cor}\label{cor_product}
Let $E$ be a $k$-coloured graph and let $\Cc$ be a complete collection of squares
in $E$ which is associative. If $\mu:E_{k,m}\rightarrow E$ and $\nu:E_{k,n}\rightarrow E$ are $\Cc$-compatible
coloured-graph morphisms such that $s(\mu)=r(\nu)$, then there exists a unique $\Cc$-compatible
coloured-graph morphism $\mu\nu:E_{k,m+n}\rightarrow E$, called the \emph{composition} of $\mu$ and
$\nu$ such that $(\mu\nu)|_{E_{k,m}}=\mu$ and $(\mu\nu)|_{E_{k,[m,m+n]}}^*=\nu$.
\end{cor}
\begin{proof}
Fix $x,y \in E^*$ such that $x$ traverses $\mu$ and $y$ traverses $\nu$.
Proposition~\ref{snake_giving_C_compatible_morphism} implies that $xy$ traverses a unique
$\Cc$-compatible coloured-graph morphism $\mu\nu$. Then $x$ traverses $(\mu\nu)|_{E_{k,m}}$, and
$y$ traverses $(\mu\nu)|_{E_{k,[m,m+n]}}^*$, so
Proposition~\ref{snake_giving_C_compatible_morphism} implies that $(\mu\nu)|_{E_{k,m}}=\mu$ and
$(\mu\nu)|_{E_{k,[m,m+n]}}^*=\nu$.

Moreover, if $\lambda$ is any other coloured-graph morphism such that $\lambda|_{E_{k,m}} = \mu$
and $\lambda|^*_{E_{k,[m,m+n]}} = \nu$ then Remark~\ref{rmk:composite traversals} shows that $xy$
traverses $\lambda$ so uniqueness in Proposition~\ref{snake_giving_C_compatible_morphism} forces
$\lambda = \mu\nu$.
\end{proof}

\begin{rmk}\label{factorisation}
Let $E$ be a $k$-coloured graph and let $\Cc$ be a complete collection of squares
in $E$ which is associative. Fix $m \leq n$ in $\NN^k$ and suppose that $\lambda:E_{k,n}\rightarrow E$ is a $\Cc$-compatible coloured-graph morphism. Corollary~\ref{cor_product} implies that $\mu :=
\lambda|_{E_{k,m}}$ and $\nu := \lambda|_{E_{k,[m,n]}}^*$ satisfy $\mu\nu=\lambda$. Suppose that
$\mu':E_{k,m}\to E$ and $\nu':E_{k,n-m}\to E$ are another two $\Cc$-compatible coloured-graph
morphisms such that $\mu^\prime\nu^\prime=\lambda$. Then $\mu^\prime=\lambda|_{E_{k,m}}=\mu$ and
$\nu^\prime=\lambda|_{E_{k,[m,n]}}^*=\nu$. So $\mu$ and $\nu$ are the unique coloured-graph
morphisms with $d(\mu) = m$, $d(\nu) = n$ and $\lambda = \mu\nu$.
\end{rmk}

\begin{cor}\label{cor_associative_product}
Let $E$ be a $k$-coloured graph and let $\Cc$ be a complete collection of squares
in $E$ which is associative. If $\lambda:E_{k,l}\rightarrow E$, $\mu:E_{k,m}\rightarrow E$ and $\nu:E_{k,n}\rightarrow
E$ are $\Cc$-compatible coloured-graph morphisms such that $s(\lambda)=r(\mu)$ and $s(\mu)=r(\nu)$,
then $\lambda(\mu\nu)=(\lambda\mu)\nu$.
\end{cor}
\begin{proof}
Fix $x_\lambda,x_\mu,x_\nu \in E^*$ such that $x_\lambda$ traverses $\lambda$, $x_\mu$ traverses
$\mu$ and  $x_\nu$ traverses $\nu$. Repeated applications of Remark~\ref{rmk:composite traversals}
show that $x_\lambda x_\mu$ traverses $\lambda\mu$. Hence $x_\lambda x_\mu x_\nu = (x_\lambda
x_\mu) x_\nu$ traverses $(\lambda\mu)\nu$. Similarly, $x_\lambda x_\mu x_\nu = x_\lambda (x_\mu
x_\nu)$ traverses $\lambda(\mu\nu)$. So Proposition~\ref{snake_giving_C_compatible_morphism}
implies that $(\lambda\mu)\nu=\lambda(\mu\nu)$.
\end{proof}

\begin{proof}[Proof of Theorem~\ref{colour to rank}]
The first statement of the theorem is precisely Corollary~\ref{cor_product}. We must check that
$\Lambda$ is a category. For composable $\mu,\nu$ we have
\[
s(\mu\nu) = (\mu\nu)(d(\mu\nu)) = (\mu\nu)|_{E_{k,[d(\mu),d(\mu)+d(\nu)]}}^* (d(\nu)) = \nu(d(\nu)) =
s(\nu),
\]
and similarly $r(\mu\nu) = r(\mu)$. Associativity of composition follows from
Corollary~\ref{cor_associative_product}. For $v \in E^0$, we have $r(\lambda_v)=\lambda_v(0)=v$ and
$s(\lambda_v)=\lambda_v(d(\lambda_v))=\lambda_v(0)=v$. Moreover, if $r(\mu)=\lambda_v$ and
$s(\nu)=\lambda_v$, then Remark~\ref{factorisation} implies that $\mu =\lambda_v\mu$ and $\nu
=\nu\lambda_v.$ Hence $\Lambda$ is a category.

Since $\NN^k$ as a category has only one object, $d$ trivially respects $r$ and $s$. It follows
immediately from the definition of composition (see Corollary~\ref{cor_product}) that $d$ respects
composition. So $d$ is a functor. Remark~\ref{factorisation} shows that $d$ satisfies the
factorisation property. So $(\Lambda,d)$ is a $k$-graph.

It remains to show that $\rho$ defines an isomorphism of $E$ with $E_\Lambda$ and that $\rho \circ
\phi \in \Cc_\Lambda$ for each $\phi \in \Cc$. The map $v \mapsto \sk{\lambda_v}$ is a bijection.
We established above that $f \mapsto \lambda_f$ is a range- and source-preserving bijection between
$c^{-1}(c_i) \subset E^1$ and $\Lambda^{e_i}$. We defined $E_\Lambda^1 = \{\sk{f} : f \in
\bigcup^k_{i=1}\Lambda^{e_i}\}$ (see Definition~\ref{dfn:E_Lambda def}). For each $f \in E^1$,
$\lambda_f$ is the unique coloured-graph morphism traversed by $f$, and $\sk{\lambda_f} \in
E^1_\Lambda$ satisfies $c_{E_\Lambda}(\sk{\lambda_f}) = c_i = c(f)$, $r(\sk{\lambda_f}) =
\sk{\lambda_f(0)} = \sk{\lambda_{r(f)}}$, and $s(\sk{\lambda_f}) = \sk{\lambda_f(e_i)} =
\sk{\lambda_{s(f)}}$. Since $\rho^1$ is bijective, the pair $(\rho^0, \rho^1) : E \to E_\Lambda$ is
an isomorphism of coloured graphs. To see that it preserves squares, fix $\psi \in \Cc$. Then $\rho
\circ \psi$ is the square $\phi_\psi$ of~\eqref{eq:phi_lambda} and hence belongs to
$\Cc_{\Lambda_{(E,\Cc)}}$ as required.
\end{proof}

\begin{proof}[Proof of Theorem~\ref{thm:uniqueness of Lambda_E}]
For $\gamma \in \Gamma$ define $\theta_{\gamma} : E_{k,m} \to E$ as in
\eqref{eq:theta(gamma)0}~and~\eqref{eq:theta(gamma)1}. Then $r(\theta_{\gamma}^1(m + v_i)) =
r(\psi^1(\sk{\gamma(m, m+e_i)})) = \psi^0(\gamma(m)) = \theta_{\gamma}^0(m)$ and similarly at the
source, so $\theta_{\gamma}$ is a graph morphism. Since $\psi^1$ preserves colour, we have
\[
c_E(\theta_{\gamma}^1(m + v_i))
    = c_E(\psi^1(\sk{\gamma(m, m+e_i)}))
    = c_{E_\Gamma}(\sk{\gamma(m, m+e_i)})
    = c_i
    = c_{E_{k, d(\gamma)}}(m + v_i),
\]
so $\theta_{\gamma}$ is a coloured-graph morphism.

To see that $\theta_\gamma$ is $\Cc$-compatible, fix a square $\alpha$ occurring in
$\theta_{\gamma}$. Then there exist $m \in \NN^k$ and $i, j \le k$ such that $\alpha(x) =
\theta_{\gamma}(x + m)$ for all $x \in E_{k, e_i + e_j}$. Let $\lambda := \gamma(m, m+e_i+e_j)$.
Then $\alpha^0(n) = \theta_{\gamma}^0(m+n) = \psi^0(\lambda(n))$ for $0 \le n \le e_i + e_j$, and
$\alpha^1(n + v_l) = \theta_{\gamma}^1(m+n+v_l) = \psi^1(\sk{\lambda(n, n+e_l)})$ whenever $n,
n+e_l \le e_i + e_j$. That is, $\alpha = \psi \circ \phi_\lambda$ where $\phi_\lambda \in
\Cc_\Gamma$ is as in Definition~\ref{dfn:E_Lambda def}. By hypothesis, that $\phi_\lambda \in
\Cc_\Gamma$ implies that $\alpha \in \Cc$, and hence $\theta_{\gamma}$ is $\Cc$-compatible. Hence
$\theta_{\gamma} \in \Lambda_{(E,\Cc)}^{d(\gamma)}$.

The assignment $\gamma \mapsto \theta_\gamma$ is a degree, range and source preserving map $\theta
: \Gamma \to \Lambda_{(E,\Cc)}$. To see that $\theta$ is injective, fix $\gamma, \gamma' \in
\Gamma$ and suppose that $\theta_{\gamma} = \theta_{\gamma'}$. Write $\gamma = \gamma_1 \dots
\gamma_n$ where each $d(\gamma_i) \in \{e_1, \dots, e_k\}$, and $\gamma' = \gamma'_1 \dots
\gamma'_n$ where each $d(\gamma'_i) = d(\gamma_i)$. For $i \le n$ define $p_i := \sum^i_{j=1}
d(\gamma_j)$. Then for each $i \le n$,
\[
\psi^1(\sk{\gamma_i}) = \theta^1_{\gamma}(p_{i-1}, p_i)
    = \theta^1_{\gamma'}(p_{i-1}, p_i\Big)
    = \psi^1(\sk{\gamma'_i}).
\]
Since $\psi^1$ is injective, it follows that $\sk{\gamma_i} = \sk{\gamma'_i}$ and hence $\gamma_i =
\gamma'_i$. So $\theta$ is injective.

To see that $\theta$ preserves composition, fix $\gamma,\gamma' \in \Gamma$ with $s(\gamma) =
r(\gamma')$, and fix paths $\psi^1(\sk{\gamma_1}) \dots \psi^1(\sk{\gamma_m})$ and
$\psi^1(\sk{\gamma'_1}) \dots \psi^1(\sk{\gamma'_n})$ which traverse $\theta_{\gamma}$ and
$\theta_{\gamma'}$. Then
\[
    \psi^1(\sk{\gamma_1}) \dots \psi^1(\sk{\gamma_m})\psi^1(\sk{\gamma'_1}) \dots \psi^1(\sk{\gamma'_n})
\]
traverses both $\theta_{\gamma\gamma'}$ and
$\theta_{\gamma}\theta_{\gamma'}$. So $\theta_{\gamma\gamma'} = \theta_{\gamma}\theta_{\gamma'}$ by
Proposition~\ref{snake_giving_C_compatible_morphism}. So $\theta$ is a functor.

To see that $\theta$ is surjective, fix $\lambda \in \Lambda_{(E,\Cc)}$ and a path $f_1 \dots f_m$
which traverses $\lambda$. Then each $f_i \in E^1$, and since $\psi^1$ is surjective, each $f_i =
\psi^1(g_i)$ for some $g_i \in E_\Gamma^1$. Each $g_i = \sk{\gamma_i}$ for some $\gamma_i \in
\Gamma$. Let $\gamma := \gamma_1 \dots \gamma_n$. Then $f_1 \dots f_m = \psi^1(\sk{\gamma_1}) \dots
\psi^1(\sk{\gamma_m})$ traverses both $\lambda$ and $\theta_{\gamma}$. So
Proposition~\ref{snake_giving_C_compatible_morphism} implies that $\theta_{\gamma} = \lambda$ and
hence $\theta$ is surjective. Thus $\theta$ is an isomorphism $\Gamma \cong \Lambda_{(E,\Cc)}$.
\end{proof}

\begin{example}
We now show how the $2$-graphs of \cite{PRW2009} fit our present mould. Let $(T,q,t,w)$ be basic
data, as in \cite[\S3]{PRW2009}. Thus $T$ is a finite hereditary subset of $\N^2$ with corners
$d_1e_1$, $d_2e_2$ satisfying $d_i\geq 1$, $q\in \N$, $t\in \ZZ/q\ZZ$ and $w:T\to \ZZ/q\ZZ$ such
that both $w(d_1e_1)$ and $w(d_2e_2)$ are nonzero. If $S$ is a subset of $\N^2$ which contains a
translate $R+n$ of a set $R$ and $f:S\to \ZZ/q\ZZ$, then we define $f|^*_{R+n}$ to be the function
$f|^*_{R+n}:R\to\ZZ/q\ZZ$ defined by $f|^*_{R+n}(i)=f(i+n)$. For $m\in \N^2$, we write
$T(m):=\bigcup\{T+n:{0\leq n\leq m}\}$.

The vertices in the $2$-graph $\Lambda(T,q,t,w)$ are the functions $v:T\to \ZZ/q\ZZ$ satisfying
$\sum_{i\in T}w(i)v(i)=t$, the paths of degree $m$ are the functions $\lambda:T(m)\to\ZZ/q\ZZ$ such
that $\lambda|^*_{T+n}\in \Lambda^0$ whenever $0\leq n\leq m$, and the range and source of
$\lambda\in \Lambda^m$ are given by $r(\lambda)=\lambda|_T$ and $s(\lambda)=\lambda|^*_{T+m}$. It
is shown in \cite[Theorem~3.4]{PRW2009} that there is a well-defined composition on
$\Lambda(T,q,t,w)$ that makes $\Lambda(T,q,t,w)$ into a $2$-graph with factorisations given by
$\lambda=(\lambda_{T(n)})(\lambda|^*_{T(m-n)+n})$.

The $2$-graph $\Lambda(T,q,t,w)$ fits our construction as follows. We define a coloured graph $E$
by $E^0=\Lambda^0$, $E^{c_i}:=\Lambda^{e_i}$, $E^1:=E^{c_1}\cup E^{c_2}$, and restricting $r,s$.
Then for each path $\lambda\in\Lambda^m$, the formulas
\[
\mu_\lambda(n)=\lambda|^*_{T+n}\ \text{ and }\ \mu_\lambda(n+v_i)=\lambda|^*_{(T\cup(T+e_i))+n}
\]
define a coloured-graph morphism, and hence a path of degree $m$ in the $2$-graph defined by $E$
with the squares given by the coloured-graph morphisms arising from the paths of degree $e_i+e_j$
in $\Lambda$. One can check, by directly constructing the inverse, that $\lambda \mapsto
\mu_\lambda$ is an isomorphism of $2$-graphs.
\end{example}

\section{Topology of path spaces}\label{sec:topology}

In \cite[Proposition~4.3]{PQR2004} the authors appeal to general category-theoretic results
\cite{Schubert:categories} to see that given a $k$-coloured graph $E$ and a complete collection of
squares $\Cc$ in $E$ which is associative, the corresponding $k$-graph $\Lambda_{(E,\Cc)}$ is
isomorphic to the quotient of the category $E^*$ under the equivalence relation $\sim$ generated by
\begin{equation}\label{eq:generating equivalences}
    \begin{split}
    \textstyle \bigcup_{n \ge 2} \{(x, y) \in E^n \times E^n: {}& \text{there exists } i < n \text{ such that } \\
        &\ x_j = y_j\text{ whenever } j \not\in \{i, i+1\} \text{ and } x_i x_{i+1} \sim_\Cc y_i y_{i+1}\}.
    \end{split}
\end{equation}
We start this section with a direct proof of this assertion by showing that each equivalence class
for $\sim$ is the set of paths which traverse some $\lambda \in \Lambda_{(E,\Cc)}$. We show that
the quotient map extends to a surjection from the space of all paths in $E$ to the space of all
paths in $\Lambda$.

We then restrict attention to $k$-graphs which are row-finite with no sources in the sense that
$0<|v\Lambda^{e_i}|<\infty$ for all $v \in \Lambda^0$ and $ i \leq k$ (see \cite{KP2000}). In this
context, the space $\Lambda^\infty$ of infinite paths in $\Lambda$ (see Remark \ref{rmk:infinite
paths} for a precise definition) --- under the topology with basic open sets $\Zz(\mu) := \{x \in
\Lambda^\infty : x(0,d(\mu)) = \mu\}$ indexed by $\mu \in \Lambda$ --- is a locally compact
Hausdorff space. Furthermore, it is the unit space of the groupoid $\Gg_\Lambda$ used to define
$C^*(\Lambda)$ in \cite{KP2000}.

We show that $\Lambda^\infty$ is the topological quotient of the space
\begin{equation}\label{eq:Ekinfty}
    \Ekinfty := \big\{x \in E^\infty : |\{i : c(x_i) = c_j\}| = \infty \text{ for each } j \le k\big\}
\end{equation}
of $k$-infinite paths in $E$. We also show that $\Ekinfty$ is a closed subspace of $E^\infty$. Lastly, we present an example which shows that these results do not necessarily hold if $\Lambda$ is not row-finite.

The following elementary lemma can be deduced from more general results in the literature (for
example \cite[Theorem~3.9]{Green:thesis}), but we provide a straightforward proof for completeness.
Recall that $q$ denotes the quotient map from $\FF^+_k$ to $\NN^k$.

\begin{lemma}\label{lem:shuffle}
Fix $w, w' \in \FF^+_k$ and suppose that $q(w) = q(w')$. Then there is a finite sequence
$(w^i)^m_{i=1}$ in $\FF^+_k$ such that $w^1 = w$, $w^m = w'$, and for each $i < m$ there exists
$j_i < |w|$ such that $w^i_l = w^{i+1}_l$ for $l \not\in \{j_i, j_i+1\}$, $w^i_{j_i} =
w^{i+1}_{j_i+1}$ and $w^i_{j_i+1} = w^{i+1}_{j_i}$.
\end{lemma}
\begin{proof}
The result is trivial if $|w| = 0$. Suppose $|w| \ge 1$ and the result holds for words of length
$|w| - 1$. Since $q(w) = q(w')$ there exists $j$ such that $w_j = w'_1$. Let
\begin{align*}
w^2 &= w_1 \dots w_{j-2} w_j w_{j-1} w_{j+1} \dots w_{|w|}, \\
w^3 &= w_1 \dots w_j w_{j-2} w_{j-1} w_{j+1} \dots w_{|w|}, \\
&\hskip1ex \vdots \\
w^j &= w_j w_1 \dots w_{j-2} w_{j-1} w_{j+1} \dots w_{|w|}.
\end{align*}
Let $x = w_1 \dots w_{j-2} w_{j-1} w_{j+1} \dots w_{|w|}$ and $x' = w'_2 \dots w'_{|w|}$. Then $w^j = w'_1 x$, $w' = w'_1 x'$, $q(x) = q(x')$ and $|x| = |w| - 1$. Apply the inductive hypothesis to $x$ and $x'$ to obtain $x^1, \dots, x^n$. The sequence $w^1, \dots w^j, w'_1 x^2, \dots, w'_1 x^n$ does the job.
\end{proof}

\begin{prop}\label{prp:quotient category}
Let $E$ be a $k$-coloured graph and let $\Cc$ be a complete collection of squares in $E$ which is
associative. Let $\sim$ be the equivalence relation on $E^*$ generated by~\eqref{eq:generating
equivalences}. For $x,y \in E^*$, we have $x \sim y$ if and only if $x$ and $y$ traverse the same
$\Cc$-compatible graph morphism $\lambda$. The structure maps $s([x]) := s(x)$, $r([x]) := r(x)$,
$d([x]) := q(c(x))$ and $[x][y] := [xy]$ are well-defined on $E^*{/}{\sim}$, and under these
operations $E^*{/}{\sim}$ is a $k$-graph which is isomorphic to $\Lambda_{(E,\Cc)}$.
\end{prop}
\begin{proof}
For a pair $(x,y)$ as in~\eqref{eq:generating equivalences}, we have $r(x) = r(y)$, $s(x) = s(y)$,
and $q(c(x)) = q(c(y))$, so the formulas $s([x]) := s(x)$, $r([x]) := r(x)$ and $d([x]) := q(c(x))$
are well-defined.

If $x \sim y$,  then there is a finite sequence of pairs $(x^l, x^{l+1})$, $1 \le l \le m-1$, each
of the form described in~\eqref{eq:generating equivalences} such that $x^1 = x$ and $x^m = y$. So
it suffices to fix $(x,y)$ as in~\eqref{eq:generating equivalences} and show that $x$ and $y$
traverse the same $\Cc$-compatible coloured-graph morphism. For this, let $\phi$ be the square in
$\Cc$ traversed by $x_i x_{i+1}$ and hence also by $y_i y_{i+1}$. By
Proposition~\ref{snake_giving_C_compatible_morphism}, $x_1\dots x_{i-1} = y_1\dots y_{i-1}$
traverses a unique $\Cc$-compatible morphism $\mu$ and $x_{i+2}\dots x_n = y_{i+2}\dots y_n$
traverses a unique $\Cc$-compatible morphism $\nu$. By Corollary~\ref{cor_product}, there is a
unique $\Cc$-compatible $\lambda = \mu\phi\nu$ which agrees, upon restriction, with each of $\mu$,
$\phi$ and $\nu$. Each of $x$ and $y$ traverse this $\lambda$.

Now suppose that $x$ and $y$ traverse a common $\Cc$-compatible morphism $\lambda$. Then in
particular $q(c(x)) = q(c(y))$. By Lemma~\ref{lem:shuffle} there is a finite sequence
$(w^i)^m_{i=1}$ in $\FF^+_k$ such that $w^1 = c(x)$, $w^m = c(y)$, and for each $i \le m-1$ there
exists $j_i < |x|$ such that $w^i_l = w^{i+1}_l$ for $l \not\in \{j_i, j_i+1\}$, and $w^i_{j_i} =
w^{i+1}_{j_i+1}\quad\text{ and } w^i_{j_i+1} = w^{i+1}_{j_i}$. For each $i$, let $z^i$ be the path
which traverses $\lambda$ such that $c(x^i) = w^i$, and for each $i \le m$ and $l \le |x|$, let
$p^i_l := q(c(x^i_1 \cdots x^i_l))$. Then for each $i \le m$, both $x^i_1 \dots x^i_{j_i-1}$ and
$x^{i+1}_1 \dots x^{i+1}_{j_i-1}$ traverse $\lambda\big(0, p^i_{j_i-1})\big)$, so they are equal,
and likewise $x^i_{j_i+2} \dots x^i_{|x|} = x^{i+1}_{j_i+2} \dots x^{i+1}_{|x|}$. Moreover each of
$x^i_{j_i} x^i_{j_i+1}$ and $x^{i+1}_{j_i} x^{i+1}_{j_i+1}$ traverses $\lambda|^*_{[p^i_{j_i-1},
p^i_{j_i+1}]}$, which, since $\lambda$ is $\Cc$-compatible, belongs to $\Cc$. Thus the pair $(x^i,
x^{i+1})$ is a pair of paths as in~\eqref{eq:generating equivalences}, and it follows that $x \sim
y$ as required.

By the preceding two paragraphs, the assignment $\rho : \lambda \mapsto [x]$ for any $x$ which
traverses $\lambda$ is a well-defined bijection from $\Lambda_{(E,\Cc)}$ to $E^*{/}{\sim}$ which
preserves range, source and degree. By definition of composition in $\Lambda_{(E,\Cc)}$, if $x$
traverses $\mu$ and $y$ traverses $\mu$, then $xy$ traverses $\lambda\mu$. So if $[x] = [x']$ and
$[y] = [y']$, then $x$ and $x'$ both traverse $\mu$, and $y$ and $y'$ both traverse $\nu$, so $xy$
and $x'y'$ both traverse $\mu\nu$. Thus
\[
[xy] = \rho(\mu\nu) = [x'y'],
\]
showing that the composition on $E^*/\sim$ is well-defined. So $\rho$ is a degree-preserving
bijective functor, and hence an isomorphism of $k$-graphs.
\end{proof}

We recall the $k$-graphs $\Omega_{k,m}$ described in \cite[Examples~2.2]{RSY1}. For $m \in (\NN
\cup \{\infty\})^k$, define $\Omega_{k,m}$ be the category with $\Obj(\Omega_{k,m}) = \{n \in \NN^k
: n \leq m\}$, $\Mor(\Omega_{k,m}) = \{(p,q) \in \NN^k \times \NN^k: p \leq q \le m\}$, $s(p,q) =
q$, $r(p,q) = p$ and $(p,q) (q,r) = (p,r)$. Then with $d(p,q) = q-p$, the pair $(\Omega_{k,m},d)$
is a row-finite $k$-graph. By convention, $\Omega_k = \Omega_{k, (\infty, \dots, \infty)}$. Note
that there is only one possible complete collection of squares $\Cc$ in the $k$-coloured graph
$E_{k,m}$, this collection is also associative, and the $k$-graph $\Lambda_{E_{k,m}, \Cc}$ of
Theorem~\ref{colour to rank} is isomorphic to $\Omega_{k,m}$.

\begin{rmk}\label{rmk:infinite paths}
Let $\Lambda$ be a $k$-graph. For $m \in \NN^k$, the factorisation property gives a bijection $\lambda \mapsto x_\lambda$ between $\Lambda^m$ and the set of graph morphisms from $\Omega_{k,m}$ to $\Lambda$: for $\lambda \in \Lambda$ and $p \le q \le d(\lambda)$, $x_\lambda(p,q)$ is the unique element of $\Lambda^{q-p}$ such that $\lambda = \lambda' x_{\lambda}(p,q) \lambda''$ for some $\lambda', \lambda''$. By analogy, for $m \in (\NN \cup \{\infty\})^k$, we call a $k$-graph morphism $x : \Omega_{k,m} \to \Lambda$ a \emph{path of degree $m$} in $\Lambda$, and we write $d(x)$ for $m$ and $r(x)$ for $x(0)$. We continue to denote the collection of all such paths by $\Lambda^m$. It is conventional to identify $\lambda$ with $x_\lambda$, and in particular to denote $x_{\lambda}(p,q)$ by $\lambda(p,q)$; so $\lambda = \lambda(0,p)\lambda(p,q)\lambda(q, d(\lambda))$ whenever $0 \le p \le q \le d(\lambda)$.
\end{rmk}

We shall write $W_\Lambda$ for the \emph{path space} $W_\Lambda := \bigcup_{m \in (\NN \cup
\{\infty\})^k} \Lambda^m$ of $\Lambda$.

\begin{prop}[{cf. \cite[Remarks~2.2]{KP2000}}]\label{prp:extended snake}
Let $E$ be a $k$-coloured graph and let $\Cc$ be a complete collection of squares in $E$ which is
associative. The map $x \mapsto \lambda_x$ from $E^*$ to $\Lambda = \Lambda_{(E,\Cc)}$ of
Proposition~\ref{snake_giving_C_compatible_morphism} extends uniquely to a degree-preserving map
$\pi : W_E \to W_\Lambda$ such that for $x \in W_E$ and $i \in \NN$ with $i \le |x|$,
$\pi(x)(0,d(x_1\dots x_i)) = \lambda_{x_1\dots x_i}$. Moreover, $\pi$ is surjective.
\end{prop}

\begin{rmk}
We have used the same symbol $\pi$ both for the map from $W_E$ to $W_\Lambda$ of
Proposition~\ref{prp:extended snake}, and for the map from $E_\Lambda$ to $\Lambda$ of
Definition~\ref{dfn:E_Lambda def}. This notation is consistent because Theorem~\ref{colour to rank}
yields a coloured-graph isomorphism $E \cong E_\Lambda$ which carries elements of $\Cc$ to elements
of $\Cc_\Lambda$.
\end{rmk}

\begin{proof}[Proof of Proposition~\ref{prp:extended snake}]
For $x \in W_E$ and $m \le n \le d(x)$, let $j$ be the least element of $\NN$ such that $d(x_1\dots
x_j) \ge n$, and define $\pi(x)(m,n) := \lambda_{x_1\dots x_j}|^*_{E_{k, [m,n]}}$.
Proposition~\ref{snake_giving_C_compatible_morphism} implies that for $j \le l$, we have
$\lambda_{x_1\dots x_l}|_{E_{k, d(x_1 \dots x_j)}} = \lambda_{x_1\dots x_j}$. Hence
$\pi(x)(0,d(x_1\dots x_j)) = \lambda_{x_1\dots x_j}$ for all $j \le |x|$. The factorisation
property in $\Lambda$ implies that $\pi(x)$ is a $k$-graph morphism from $\Omega_{k, d(x)}$ to
$\Lambda$. For uniqueness of $\pi$, observe that by uniqueness of factorisations in $\Lambda$, any
$y \in W_\Lambda$ such that $y(x_1\dots x_i) = \lambda_{x_1 \dots x_i}$ for all $i \le d(x)$ must
satisfy $y(m,n) = \lambda_{x_1\dots x_j}|^*_{E_{k, [m,n]}}$ whenever $d(x_1\dots x_j) \ge n$.

To see that $\pi$ is surjective first note that if $\lambda \in \Lambda$ then any path $x$ which
traverses $\lambda$ satisfies $\pi(x) = \lambda$. So fix $y \in W_\Lambda \setminus \Lambda$. Fix a
sequence $(m_j)^\infty_{j=0}$ such that $m_0 = 0$, $m_{j+1} - m_j \in \{e_1, \dots, e_k\}$ for all
$j$ and $\bigvee_{j \in \NN} m_j = d(y)$. For each $j \in \NN$ define $x_j := y(m_{j-1}, m_j) \in
E^1$. Then $x = x_1 x_2 \dots \in W_E$, and $\pi(x)(m,n) = y(m,n)$ for all $m,n$ by uniqueness of
factorisations in $\Lambda$, so $\pi(x) = y$.
\end{proof}

If $\pi : W_E \to W_\Lambda$ is the surjection of Proposition~\ref{prp:extended snake}, then
Proposition~\ref{cor_product} implies that $\pi(x)\pi(y) = \pi(xy)$ when $x$ and $y$ are finite
with $r(y) = s(x)$.

Now let $\Lambda$ be a row-finite $k$-graph with no sources. Recall that $\Lambda^\infty$ is the collection of $k$-graph morphisms from $\Omega_k$ to $\Lambda$, and $\Ekinfty$ is the collection of infinite paths in $E$ which contain infinitely many edges of each colour. Since $\pi$ is surjective, it follows that $\Lambda^\infty = \pi(\Ekinfty)$.

\begin{rmk}
Let $E$ be a $k$-coloured graph and let $\Cc$ be a complete collection of squares in $E$ which is
associative. Let $\Lambda = \Lambda_{(E,\Cc)}$ be the corresponding $k$-graph as in
Theorem~\ref{colour to rank}. Identify $\Lambda^0$ with $E^0$. Then for each $v \in \Lambda^0$ and
$i \le k$, we have $|v\Lambda^{e_i}| = |\{e \in E^1 : r(e) = v\text{ and } c(e) = c_i\}|$. Hence
$\Lambda$ is row-finite and has no sources if and only if $0 < |\{e \in E^1 : r(e) = v\text{ and }
c(e) = c_i\}| < \infty$ for all $v \in E^0$ and $i \le k$.
\end{rmk}

\begin{prop}\label{rf q cts}
Let $E$ be a $k$-coloured graph and let $\Cc$ be a complete collection of squares in $E$ which is
associative. Let $\Lambda = \Lambda_{(E,\Cc)}$ as in Theorem~\ref{colour to rank}. Suppose that
$\Lambda$ is row-finite and has no sources. Let $\pi : \Ekinfty \to \Lambda^\infty$ be the
restriction of the surjection of Proposition~\ref{prp:extended snake}. Then $U \subseteq
\Lambda^\infty$ is open if and only if $\pi^{-1}(U) \subseteq \Ekinfty$ is open.
\end{prop}
\begin{proof}
First suppose that $U$ is open in $\Lambda^\infty$, and fix $x \in \pi^{-1}(U)$. We seek a basic
open set $B_x$ in $\Ekinfty$ such that $x \in B_x \subset \pi^{-1}(U)$. Since $U$ is open, there
exists $\mu \in \Lambda$ such that $\pi(x) \in \Zz(\mu) \subset U$. Fix $n \in \NN$ such that
$q(c(x_1 \dots x_n)) > d(\mu)$. Then $\pi(x_1 \dots x_n) \in \Zz(\mu)$. Let $y_x = x_1 \dots x_n$.
Then $x \in \Zz(y_x)$. To see $\Zz(y_x) \subset \pi^{-1}(U)$, fix $y \in \Zz(y_x)$; say $y = y_x
y'$. Then $\pi(y) = \pi(x_1 \dots x_n y') = \pi(x_1\dots x_n)\pi(y')\in \Zz(\mu) \subset U$, so $y
\in \pi^{-1}(U)$ as required.

For the reverse implication, suppose that $\pi^{-1}(U)$ is open in $\Ekinfty$, and fix $\lambda \in
U$. We seek a basic open set $B_\lambda$ such that $\lambda \in B_\lambda \subset U$. Fix $x \in
E^\infty$ which traverses $\lambda$. Then $x \in \Ekinfty$, and $x \in \pi^{-1}(U)$ which is open.
Hence there exists a basic open set $B_x \in \Ekinfty$ such that $x \in B_x \subset \pi^{-1}(U)$.
So $B_x = \Zz(y_x)$ for some $y_x \in E^*$, and
\[
    \lambda = \pi(x) = \pi(y_x x') = \pi(y_x) \pi(x') \in \Zz(\pi(y_x)).
\]
To see that $\Zz(\pi(y_x)) \subset U$, let $\mu \in \Zz(\pi(y_x))$. Write $\mu = \pi(y_x)\mu'$, and
let $x_{\mu'}$ be a path in $E^*$ which traverses $\mu'$. Then $y_x x_{\mu'} \in \Zz(y_x) \subset
\pi^{-1}(U)$, which implies that $\mu = \pi(y_x x_{\mu'}) \in U$.
\end{proof}

Proposition~\ref{rf q cts} implies that when $E$ is row-finite, the topology on $\Lambda^\infty$ is
the quotient topology inherited from $\Ekinfty$ under $\pi$. In particular, $\pi$ is continuous.

As mentioned in the opening of this section, when $\Lambda$ is row-finite with no sources,
$\Lambda^\infty$ is homeomorphic to the unit space of the groupoid $\Gg_\Lambda$ of \cite{KP2000};
it is also homeomorphic to the spectrum of the commutative subalgebra of $C^*(\Lambda)$ spanned by
the projections $s_\lambda s^*_\lambda$ (see \cite{pp_Webster2011} and the opening of
Section~\ref{sec:simple}). If $\Lambda$ is not row-finite, this is no longer the case:
$\Lambda^\infty$ need not be locally compact. To see this, suppose that $\Lambda$ is the $1$-graph
with one vertex and infinitely many edges $\{f_i : i \in \NN\}$. Then given any $x \in
\Lambda^\infty$, any neighbourhood of $x$ contains $\Zz(x_1\dots x_n)$ for some $n$, and the cover
$\Zz(x_1\dots x_n) = \bigcup_{i=1}^\infty \Zz(x_1\dots x_n f_i)$ has no finite subcover.

Instead, given a finitely aligned $k$-graph, let
\[\begin{split}
    \Lambda^{\le \infty} := \{x \in W_\Lambda :&\text{ there exists } n \le d(x)\text{ such that } \\
        &(n \le p \le d(x)\text{ and }p_i = d(x)_i)\text{ implies } x(p)\Lambda^{e_i} = \emptyset\}
\end{split}\]
as in \cite{RSY2004}. Endow $W_\Lambda$ with the topology with basic open sets $\Zz(\mu \setminus
G) := \Zz(\mu) \setminus \Big(\bigcup_{\lambda \in G} \Zz(\mu\lambda)\Big)$, where $\mu$ ranges
over $\Lambda$ and $G$ ranges over all finite subsets of $s(\mu)\Lambda$. Then the unit space of
$\partial\Lambda$ of the groupoid $\Gg_\Lambda$ constructed in \cite{FMY2005} is the closure of
$\Lambda^{\le \infty}$ in $W_\Lambda$; this is also homeomorphic to the spectrum of
$\clsp\{s_\lambda s^*_\lambda : \lambda \in \Lambda\}$ \cite{pp_Webster2011}. So the natural
question to ask for finitely aligned $k$-graphs is whether the quotient topology on
$\partial\Lambda$ is the same as its standard topology. The next example shows that it is not.

\begin{example}
Let $E$ be the $2$-coloured graph pictured below.
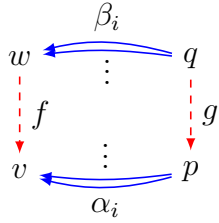
\[\begin{tikzpicture}[xscale=0.75,yscale=0.5,>=latex,semithick]\label{2graph ex}
    \node (p) at (3,0) {$p$};
    \node (v) at (0,0) {$v$};
    \node (w) at (0,3) {$w$}
        edge[->,dashed, red] node[black,auto] {$f$}(v);
    \node (q) at (3,3) {$q$}
        edge[->,dashed, red] node[black,auto] {$g$} (p);
    \draw [->, blue]   (q) to [out=160,in=20] node[black,auto,swap] {$\beta_i$} (w);
    \draw [->, blue]   (q) to [out=170,in=10] node[black,yshift=-5pt] {$\vdots$} (w);
    \draw [->, blue]   (p) to [out=-160,in=-20] node[black,auto] {$\alpha_i$} (v);
    \draw [->, blue]   (p) to [out=-170,in=-10] node[black,auto,swap] {$\vdots$} (v);
\end{tikzpicture}
\]

Let $\Cc$ be the collection of graph morphisms $\lambda_i:E_{2,(1,1)} \to E$ such that $\alpha_i g$
and $f \beta_i$ both traverse $\lambda_i$ for each $i$. This is a complete collection of squares in
$E$. Since $E$ has only two colours $\Cc$ is associative. Let $\Lambda$ be the $2$-graph
constructed from $(E,\Cc)$ as in Theorem~\ref{colour to rank}, and let $\pi : W_E \to W_\Lambda$ be
the surjection of Proposition~\ref{prp:extended snake}. Then $v\Lambda^{\le \infty} =
\Lambda^{(1,1)} = \{\lambda_i : i \in \NN\}$ and $\pi(\alpha_i g) = \lambda_i = \pi(f \beta_i)$ for
all $i$.

We claim that $\alpha_i g \to v$ in $W_E$ but that $\lambda_i \to f$ in $W_\Lambda$. To see that
$\alpha_i g \to v$ in $W_E$, fix a basic open set $\Zz(y \setminus F) \subset E^*$ containing $v$.
Then $y = v$. Since $F$ is finite, there are only finitely many $i$ such that either $\alpha_i$ or
$\alpha_i g$ belongs to $F$. Let $N_0:=\max\{ i : \alpha_i \in F\text{ or } \alpha_i g \in F\}$.
Then $\alpha_n g \in \Zz(v \setminus F)$ for all $n \ge N_0$, whence $\alpha_i g \to v$ as $i \to
\infty$.

To see that $\lambda_i \to \pi(f)$ in $W_\Lambda$, fix a basic open set $\Zz(\mu \setminus G)
\subset \Lambda$ containing $f$. Then either $\mu = \pi(f)$ or $\mu = v$. We show that $\lambda_i
\in \Zz(\mu \setminus G)$ for large $i$. First suppose that $\mu = \pi(f)$. Then $G$ is a finite
collection of paths of the form $\pi(\beta_i \nu)$. Let $N_1 = \max\{i : \pi(\beta_i \nu) \in
G\text{ for some }\nu\}$. Then $\lambda_n = \pi(f \beta_n) \in \Zz(f \setminus G)$ for all $n \geq
N_1$. Now suppose that $\mu = v$. Since $G$ does not contain $\pi(f)$, it is a finite subset of
$\{\pi(\alpha_i), \lambda_i : i \in \NN\}$. Let $N_2 = \max\{i: \pi(\alpha_i) \in G\text{ or }
\lambda_i \in G\}$. Then $\lambda_n \in \Zz(v \setminus G)$ for all $n \ge N_2$. Hence $\lambda_i
\to f$ as $i \to \infty$.

We now have $\pi(\lim \alpha_i g) = \pi(v) \ne \pi(f) = \lim \lambda_i = \lim \pi(\alpha_i g)$, so
$\pi$ is not continuous.
\end{example}

\section{Simplicity of \texorpdfstring{$C^*$}{C*}-algebras of higher-rank graphs}\label{sec:simple}

Suppose that $\Lambda$ is a $k$-graph which is row-finite and has no sources. For such $\Lambda$, a
\emph{Cuntz-Krieger $\Lambda$-family} in a $C^*$-algebra $B$ consists of partial isometries
$\{t_\lambda:\lambda\in\Lambda\}$ satisfying the \emph{Cuntz-Krieger relations} \cite{KP2000}:
\begin{enumerate}
\renewcommand{\theenumi}{CK\arabic{enumi}}
\item $\{t_v:v\in\Lambda^0\}$ are mutually orthogonal projections;\label{ck1}
\item $t_{\lambda\mu}=t_\lambda t_\mu$ for all $\lambda,\mu\in\Lambda$ with
    $s(\lambda)=r(\mu)$;\label{ck2}
\item $t_\lambda^*t_\lambda=t_{s(\lambda)}$ for all $\lambda \in \Lambda$; and\label{ck3}
\item $t_v=\sum_{\lambda\in v\Lambda^m}t_\lambda t_\lambda^*$ for all $v\in\Lambda^0$ and
    $m\in\NN^k$.\label{ck4}
\end{enumerate}
The graph $C^*$-algebra $C^*(\Lambda)$ is the $C^*$-algebra generated by a universal Cuntz-Krieger
$\Lambda$-family $\{s_\lambda:\lambda\in\Lambda\}$; it follows from \cite[Proposition~2.11]{KP2000}
that each vertex projection $s_v$ is nonzero.

As in \cite{RS2007}, we say that $\Lambda$ is \emph{aperiodic} if for every vertex $v\in\Lambda^0$
 and each pair $m\ne n\in\NN^k$ there is a path $\lambda\in v\Lambda$ such that $d(\lambda)\geqslant m\vee n$ and
\[
\lambda(m,m+d(\lambda)-(m\vee n))\ne \lambda(n,n+d(\lambda)-(m\vee n)).
\]
Lemma~3.2 of \cite{RS2007} implies that this formulation of aperiodicity in terms of finite paths
is equivalent to the aperiodicity condition used in \cite{KP2000}. So the next theorem follows from
\cite[Theorem~4.6]{KP2000}.

\begin{theorem}[The Cuntz-Krieger uniqueness theorem]\label{thm_Cuntz_Krieger_Uniqueness}
Let $\Lambda$ be a row-finite, aperiodic $k$-graph with no sources. Suppose that
$\{t_\lambda:\lambda\in\Lambda\}$ is a Cuntz-Krieger $\Lambda$-family, and let $\pi$ be the
homomorphism of $C^*(\Lambda)$ such that $\pi(s_\lambda)=t_\lambda$ for all $\lambda \in \Lambda$.
If each $t_v$ is nonzero, then $\pi$ is faithful.
\end{theorem}

The proof of this theorem in \cite{KP2000} uses a groupoid model for $C^*(\Lambda)$. Here we
outline a direct proof that flows from the finite-path formulation of aperiodicity via the
following lemma.

\begin{lemma}\label{lem_there_exists_lambda}
Let $(\Lambda,d)$ be an aperiodic $k$-graph with no sources. Suppose that $v\in\Lambda^0$ and
$l\in\NN^k$. Then there exists $\lambda\in \Lambda$ such that $r(\lambda)=v$, $d(\lambda)\geq l$
and
\begin{equation}\label{proplambda}
\alpha,\beta\in \Lambda v,\ d(\alpha),d(\beta)
\leq l\ \text{and}\ \alpha\not=\beta\Longrightarrow(\alpha\lambda)(0,d(\lambda)) \not=
(\beta\lambda)(0,d(\lambda)).
\end{equation}
\end{lemma}
\begin{proof}
We list pairs $(m,n)$ of distinct elements of $\N^k$ with $0\le m,n \le l$ as $\{(m^{(i)},n^{(i)})
: 1 \le i \le p\}$. Then an induction on $i$ shows that there exist $\mu_i$ and $l^{(i)}\in \N^k$
such that $r(\mu_1)=v$, $r(\mu_i)=s(\mu_{i-1})$ for $i\geq 1$, $d(\mu_i)=(m^{(i)}\vee
n^{(i)})+l^{(i)}$, and $\mu_i(m^{(i)},m^{(i)}+l^{(i)})\ne \mu_i(n^{(i)},n^{(i)}+l^{(i)})$. We now
choose an arbitrary path $\lambda'$ with $d(\lambda')\geq l$ and $r(\lambda')=s(\mu_p)$, and claim
that $\lambda:=\mu_1\mu_2\cdots\mu_p\lambda'$ has the required properties. We trivially have
$d(\lambda)\geq l$.

Suppose that $\alpha$ and $\beta$ are distinct paths with source $v$ and $d(\alpha)\vee d(\beta)\le
l$. If $d(\alpha)=d(\beta)=d$, say, then the initial segments $\alpha=(\alpha\lambda)(0,d)$ and
$\beta=(\beta\lambda)(0,d)$ are not equal, and $\alpha\lambda\not=\beta\lambda$. So suppose that
$d(\alpha)\not=d(\beta)$, say $(d(\alpha),d(\beta))=(m^{(i)},n^{(i)})$. Let $d:=
\sum_{j=1}^{i-1}d(\mu_j)$. Then
\[
(\alpha\lambda)(d(\alpha) + d + n^{(i)}, d(\alpha) + d + n^{(i)} + l^{(i)})
    =\mu_i(n^{(i)},n^{(i)}+l^{(i)})
\]
is not the same as $\mu_i(m^{(i)},m^{(i)}+l^{(i)})=(\beta\lambda)(d(\beta) + d + m^{(i)}, d(\beta)
d + m^{(i)} + l^{(i)})$. Since $d(\beta) + d + m^{(i)} = d + m^{(i)} + n^{(i)} = d(\alpha) + d +
n^{(i)}$, it follows that
\[
(\alpha\lambda)(d + m^{(i)} + n^{(i)}, d + m^{(i)} + n^{(i)} + l^{(i)})
    \not= (\beta\lambda)(d + m^{(i)} + n^{(i)}, d + m^{(i)} + n^{(i)} + l^{(i)}).
\]
The presence of the factor $\lambda'$ forces $d(\lambda)\geq d + m^{(i)} + n^{(i)} + l^{(i)}$, so
$(\alpha\lambda)(0,d(\lambda))\not=(\beta\lambda)(0,d(\lambda))$, as required.
\end{proof}

\begin{prop}\label{prp:norm-decreasing}
Suppose that $\Lambda$ is a row-finite aperiodic $k$-graph with no sources, and let $\{t_\lambda :
\lambda \in \Lambda\}$ be a Cuntz-Krieger $\Lambda$-family in a $C^*$-algebra $B$ such that $t_v
\not= 0$ for all $v \in \Lambda^0$. Let $F$ be a finite subset of $\Lambda$ and let $a : (\mu,\nu)
\mapsto a_{\mu,\nu}$ be a $\CC$-valued function on $F \times F$ such that $s(\mu) = s(\nu)$ whenever $a_{\mu,\nu} \not= 0$. Then
\[
    \Big\| \sum_{\mu,\nu \in F} a_{\mu,\nu} t_\mu t^*_\nu \Big\| \ge \Big\|\sum_{\mu,\nu \in F, d(\mu) = d(\nu)} a_{\mu,\nu} t_\mu t^*_\nu \Big\|.
\]
\end{prop}
\begin{proof}
Let $a := \sum_{\mu,\nu \in F} a_{\mu,\nu} t_\mu t^*_\nu$ and let $a_0 := \sum_{\mu,\nu \in F,
d(\mu) = d(\nu)} a_{\mu,\nu} t_\mu t^*_\nu$. Define $n := \bigvee_{\mu \in F} d(\mu)$, and let $G
:= \bigcup_{\mu \in F} F s(\mu)\Lambda^{n - d(\mu)}$. So if $\mu,\nu \in F$ with $s(\mu) = s(\nu)$
and $d(\mu\alpha) = n$, then $\mu\alpha, \nu\alpha \in G$. By applying~\eqref{ck4} at $s(\mu)$ for
each $\mu,\nu \in F$, we can express
\[
a = \sum_{\mu,\nu \in G} b_{\mu,\nu} t_\mu t^*_\nu\quad\text{ and }\quad
    a_0 = \sum_{\mu,\nu \in G, d(\mu) = d(\nu)} b_{\mu,\nu} t_{\mu} t^*_{\nu},
\]
where $b_{\mu,\nu} \not= 0$ implies $d(\mu) = n$ and $s(\mu) = s(\nu)$.

For each $v \in s(G)$, apply Lemma~\ref{lem_there_exists_lambda} with $l = \bigvee_{\nu \in G}
d(\nu)$ to find $\lambda_v \in v\Lambda$ such that $d(\lambda) \ge l$ and
\[
    (\alpha\lambda_v)(0, l) \not= (\beta\lambda_v)(0,l)\text{ for distinct }\alpha,\beta \in Gv,
\]
and let $Q_v := \sum_{\alpha \in Gv, d(\alpha) = n} t_{\alpha\lambda_v} t^*_{\alpha\lambda_v}$.
Then~(\ref{ck3}) implies that the $Q_v$ are mutually orthogonal projections. Hence
\begin{equation}\label{eq:Q.Q norm decreasing}
\Big\|\sum_{v \in s(G)} Q_v a Q_v\Big\| \le \|a\|.
\end{equation}

We show that
\begin{equation}\label{eq:QaQ=Qa0Q}
\sum_{v \in s(G)} Q_v a Q_v = \sum_{v \in s(G)} Q_v a_0 Q_v.
\end{equation}
For $\mu,\nu \in G$ with $s(\mu) = s(\nu)$ and $d(\mu) = n$, a quick calculation using~(\ref{ck4})
gives
\begin{equation}\label{eq:Qmult calculation}
Q_v t_\mu t^*_\nu = \delta_{v, s(\mu)} t_{\mu\lambda_{s(\mu)}} t^*_{\nu\lambda_{s(\mu)}}.
\end{equation}
Suppose $d(\mu) \not= n$ and fix $\alpha \in G \cap \Lambda^n$. Then
$(\alpha\lambda_{s(\alpha)})(0,l) \not= (\nu\lambda_{s(\nu)})(0,l)$, and hence
$t^*_{\nu\lambda_{s(\mu)}} t_{\alpha\lambda_{s(\alpha)}} = 0$. This and~\eqref{eq:Qmult
calculation} give $Q_v t_\mu t^*_\nu Q_v = 0$ for all $v$, and~\eqref{eq:QaQ=Qa0Q} follows.

Finally we show that
\begin{equation}\label{eq:Q.Q isometric}
\Big\|\sum_{v \in s(G)} Q_v a_0 Q_v\Big\| = \|a_0\|.
\end{equation}

Routine calculations using the Cuntz-Krieger relations and that the $t_v$ are all nonzero show that
$\{t_\mu t^*_\nu : \mu,\nu \in G \cap \Lambda^n, s(\mu) = s(\nu)\}$ is a family of nonzero matrix
units spanning an isomorphic copy of $\bigoplus_{v \in s(G)} M_{Gv \cap \Lambda^n}(\CC)$, and that
$\{t_{\mu\lambda_{s(\mu)}} t^*_{\nu\lambda_{s(\nu)}} : \mu,\nu \in G \cap \Lambda^n, s(\mu) =
s(\nu)\}$ is a family of nonzero matrix units for the same finite-dimensional $C^*$-algebra. Hence
$t_\mu t^*_\nu \mapsto t_{\mu\lambda_{s(\mu)}} t^*_{\nu\lambda_{s(\nu)}}$ determines an isomorphism
of finite-dimensional subalgebras of $C^*(\Lambda)$, so is isometric. Calculations
like~\eqref{eq:Qmult calculation} show that $\sum_{v \in s(G)} Q_v t_\mu t^*_\nu Q_v =
t_{\mu\lambda_{s(\mu)}} t^*_{\nu\lambda_{s(\mu)}}$ whenever $\mu,\nu \in G \cap \Lambda^n$ with
$s(\mu) = s(\nu)$, and~\eqref{eq:Q.Q isometric} follows.

Combining~\eqref{eq:QaQ=Qa0Q}, \eqref{eq:Q.Q norm decreasing}~and~\eqref{eq:Q.Q isometric} proves
the Proposition.
\end{proof}

For the following proof, recall from the opening of \cite[Section~3]{KP2000} that there is a
strongly continuous action $\gamma : \TT^k \to \Aut(C^*(\Lambda))$ characterised by
\[
\gamma_z(s_\lambda) = z^{d(\lambda)} s_\lambda = z_1^{d(\lambda)_1}z_2^{d(\lambda)_2} \cdots z_k^{d(\lambda)_k} s_\lambda
\]
for all $\lambda \in \Lambda$. Averaging over this action yields a faithful conditional expectation
$\Phi : C^*(\Lambda) \to \clsp\{s_\mu s^*_\nu : d(\mu) = d(\nu)\}$ (see \cite[Lemma~3.3]{KP2000})
such that $\Phi(s_\mu s^*_\nu) = \delta_{d(\mu), d(\nu)} s_\mu s^*_\nu$ for all $\mu,\nu \in
\Lambda$.

\begin{proof}[Proof of Theorem~\ref{thm_Cuntz_Krieger_Uniqueness}]
Follow the first paragraph of the proof of \cite[Theorem~3.4]{KP2000} to see that $\pi$ is
injective on $C^*(\Lambda)^\gamma = \clsp\{s_\mu s^*_\nu : d(\mu) = d(\nu)\}$.

Proposition~\ref{prp:norm-decreasing} implies that the formula
\[
\sum_{\mu,\nu \in F} a_{\mu,\nu} t_\mu t^*_\nu \mapsto \sum_{\mu,\nu \in F, d(\mu) = d(\nu)} a_{\mu,\nu} t_\mu t^*_\nu
\]
is well defined on finite linear combinations (if two linear combinations are equal,
Proposition~\ref{prp:norm-decreasing} implies that the norm of the difference of their images is
zero), and norm-decreasing, and hence extends by continuity to a linear map $\Psi :
\pi_t(C^*(\Lambda)) \to \clsp\{t_\mu t^*_\nu : d(\mu) = d(\nu)\}$ such that $\Psi(t_\mu t^*_\nu) =
\delta_{d(\mu), d(\nu)} t_\mu t^*_\nu$.

To complete the proof, we argue as in the last two lines of the proof of
\cite[Theorem~3.4]{KP2000}: Let $\Phi$ be the faithful conditional expectation on $C^*(\Lambda)$
described above. By linearity and continuity, $\pi \circ \Phi = \Psi \circ \pi$. Suppose that
$\pi(a) = 0$. Then $\Psi(\pi(a^*a)) = 0$ and hence $\pi(\Phi(a^*a)) = 0$. Since $\pi$ is injective
on $C^*(\Lambda)^\gamma$, it follows that $\Phi(a^*a) = 0$. Since $\Phi$ is a faithful expectation,
we then have $a^*a = 0$ and hence $a = 0$.
\end{proof}

Let $\Lambda$ be a row-finite graph without sources. As in \cite{LS2010}, we say that $\Lambda$ is
\emph{cofinal} if for every pair $v,w \in \Lambda^0$ there exists $n \in \NN^k$ such that $v\Lambda
s(\lambda) \not= \emptyset$ for all $\lambda \in w\Lambda^n$.

\begin{rmk}
For row-finite graphs without sources, \cite[Proposition~A.2]{LS2010} implies that this notion of
cofinality is equivalent to \cite[Definition~3.3]{LS2010}, and hence by \cite[Theorem~5.1]{LS2010}
to the usual one involving infinite paths.
\end{rmk}

Modulo the different formulation of cofinality, the following characterisation of simplicity
appeared in \cite{RS2007}, and was generalised to locally convex $k$-graphs in \cite{RS2009} and
finitely aligned $k$-graphs in \cite{LS2010,pp_Shotwell2008}.

\begin{theorem}\label{thm:simple}
Let $\Lambda$ be a row-finite $k$-graph with no sources. Then $C^*(\Lambda)$ is simple if and only
if $\Lambda$ is both aperiodic and cofinal.
\end{theorem}

In the proof we use the infinite-path representation. By \cite[Proposition~2.3]{KP2000}, for $x \in
\Lambda^\infty$, $\lambda \in \Lambda r(x)$ and $n \in \NN^k$, there are unique elements
$\sigma^n(x)$ and $\lambda x$ of $\Lambda^\infty$ such that
\[
\sigma^n(x)(p,q) = x(n+p, n+q)\ \text{and}\
(\lambda x)(p,q) = (\lambda x(0, q))(p,q)
\]
for $p \le q \in \NN^k$. Let $\{\xi_x : x \in \Lambda^\infty\}$ be the usual orthonormal basis for
$\ell^2(\Lambda^\infty)$. Then for each $\lambda\in \Lambda$ there is a partial isometry
$S_\lambda$ on $\ell^2(\Lambda^\infty)$ such that $S_\lambda\xi_x=\xi_{\lambda x}$, and
$\{S_\lambda : \lambda \in \Lambda\}$ is a Cuntz-Krieger $\Lambda$-family which gives a
representation $\pi_S$ of $C^*(\Lambda)$ on $\ell^2(\Lambda^\infty)$.

The following lemma is a special case of the implication \mbox{(ii)$\;\implies\;$(i)} in
\cite[Theorem~5.1]{LS2010}, but the proof simplifies significantly in our setting.

\begin{lemma}\label{lem:what is cofinality}
Let $\Lambda$ be a row-finite $k$-graph with no sources which is not cofinal. Then there exist a
vertex $v \in \Lambda^0$ and an infinite path $x \in \Lambda^\infty$ such that $v\Lambda x(n) =
\emptyset$ for all $n \in \NN^k$.
\end{lemma}

\begin{proof}
Since $\Lambda$ is not cofinal, there exist $v,w \in \Lambda^0$ such that, for each $n \in \NN^k$,
there exists $\lambda \in w\Lambda^n$ with $v\Lambda s(\lambda) = \emptyset$. Choose $n^{(i)}\to
\infty$ in $\NN^k$, and $\lambda_i\in w\Lambda^{n^{(i)}}$ such that $v\Lambda s(\lambda_i)
=\emptyset$. Let $1_k = (1, \dots, 1) \in \NN^k$. Since $\Lambda$ is row-finite, there exists
$\mu_1 \in v\Lambda^{1_k}$ such that $S_1 := \{j \in \NN : \lambda_j(0, 1_k) = \mu_1\}$ is
infinite. An induction argument now shows that there is a sequence $(\mu_i)^\infty_{i=1}$ in
$v\Lambda$ such that for every $i \ge 2$, we have $\mu_i \in \mu_{i-1} \Lambda^{1_k}$, and $S_i :=
\{j \in S_{i-1} : \lambda_j(0, i\cdot 1_k)\}$ is infinite. In particular, for any $i \in \NN$ and
$j \in S_i$ that $v\Lambda s(\lambda_j) =\emptyset$ forces $v\Lambda s(\mu_i) =\emptyset$. Since
$d(\mu_i) \to (\infty, \dots, \infty)$, \cite[Remarks~2.2]{KP2000} imply that there is an infinite
path $x$ such that $x(0,d(\mu-i))=\mu_i$ for all $i$. Now since $v\Lambda x(d(\mu_i))=v\Lambda
s(\mu_i)=\emptyset$ for all $i$, we have $v\Lambda x(n)=\emptyset$ for all $n$, so the infinite
path $x$ has the required properties.
\end{proof}

\begin{proof}[Proof of Theorem~\ref{thm:simple}]
First suppose that $\Lambda$ is aperiodic and cofinal, and let $I$ be a nonzero ideal in
$C^*(\Lambda)$. To see that $I = C^*(\Lambda)$, we fix $\mu\in\Lambda$ and aim to show that $s_\mu$
belongs to $I$. Since $\Lambda$ is aperiodic,  the Cuntz-Krieger uniqueness theorem
(Theorem~\ref{thm_Cuntz_Krieger_Uniqueness}) implies that $I$ contains a vertex projection $s_v$.
Applying cofinality with this $v$ and $w = s(\mu)$ gives $n \in \NN^k$ such that for each $\lambda
\in s(\mu)\Lambda^n$ there exists $\nu_\lambda \in v\Lambda s(\lambda)$. Then
\[
s_\mu
    = \sum_{\lambda \in s(\mu)\Lambda^n} s_{\mu\lambda} s_{s(\lambda)} s_\lambda^*
    = \sum_{\lambda \in s(\mu)\Lambda^n} s_{\mu\lambda} (s^*_{\nu_\lambda} s_v s_{\nu_\lambda}) s_\lambda^*
    \in I,
\]
as required.

For the other direction, we first suppose that $\Lambda$ is not aperiodic, so that there exist $v
\in \Lambda^0$ and distinct $m,n \in \NN^k$ such that, for every $\lambda\in v\Lambda$ with
$d(\lambda)\geq m\vee n$,
\begin{equation}\label{eq:not aperiodic}
    \lambda(m, m+d(\lambda) - (m \vee n)) = \lambda(n, n + d(\lambda) - (m \vee n)).
\end{equation}
Then for every $x \in v\Lambda^\infty$ and $l \in \NN^k$, we can apply \eqref{eq:not aperiodic} to
$\lambda = x(0, (m \vee n) + l)$ and deduce that $x(m, m+l) = x(n, n+l)$, whence $\sigma^m(x) =
\sigma^n(x)$.

We now fix $\lambda \in v\Lambda^{m \vee n}$, let $\mu = \lambda(0,m)$ and $\nu = \lambda(0,n)$,
and aim to prove that $a := s_\lambda s^*_\lambda -s_\mu s^*_\nu s_\lambda s^*_\lambda$ is nonzero
and belongs to $\ker\pi_S$; since we know that $\ker \pi_S$ does not contain any vertex
projections, this will prove that $\ker\pi_S$ is a nontrivial ideal. To see that $\pi_S(a)=0$, fix
$x\in\Lambda^\infty$ and compute
\begin{equation}\label{eq:equals proj}
\pi_S(a)\xi_x = (S_\lambda S^*_\lambda-S_\mu S^*_\nu S_\lambda S^*_\lambda)\xi_x.
\end{equation}
If $x(0,d(\lambda))\not=\lambda$, then \eqref{eq:equals proj} vanishes. If $x(0,d(\lambda)) =
\lambda$, then $x \in v\Lambda^\infty$, the argument in the previous paragraph gives $\sigma^m(x) =
\sigma^n(x)$, and hence
\[
\nu\sigma^n(x) = x = \mu\sigma^m(x) = \mu \sigma^n(x),
\]
which implies $S_\mu S^*_\nu \xi_x = \xi_{\mu\sigma^n(x)} = \xi_x$ and $\pi_S(a)\xi_x=0$. Thus
$\pi_S(a)=0$.

To see that $a\not=0$, we choose $z \in \TT^k$ such that $z^{m-n} = -1$. Then $\gamma_z(s_\mu
s^*_\nu s_\lambda s^*_\lambda)  = -s_\mu s^*_\nu s_\lambda s^*_\lambda$, and hence
\[
\pi_S(a + \gamma_z(a)) = \pi_S(2s_\lambda s^*_\lambda) = 2S_\lambda S^*_\lambda \not= 0,
\]
forcing $a \not = 0$. Thus $C^*(\Lambda)$ is not simple.

Now suppose that $\Lambda$ is not cofinal. By Lemma~\ref{lem:what is cofinality}, there exist $v
\in \Lambda^0$ and $x \in \Lambda^\infty$ such that $v\Lambda x(n) = \emptyset$ for all $n \in
\NN^k$. Let
\[
[x]_\sigma := \{y \in \Lambda^\infty :\text{ there exist } p,q \in \NN^k
        \text{ such that } \sigma^p(x) = \sigma^q(y)\}.
\]
We claim that
\begin{equation}\label{eq:Wx property}
y \in [x]_\sigma\Longrightarrow r(y) \not= v.
\end{equation}
To see this, fix $y \in [x]_\sigma$ and $p,q \in \NN^k$ such that $\sigma^p(x) = \sigma^q(y)$. Then
$x(p)= y(q)$ and hence $v\Lambda y(q) = \emptyset$ by choice of $x$. In particular, $y(0,q)
\not\in v\Lambda y(q)$, so $r(y) \not= v$, as claimed.

We now consider the subspace $\Hh_x := \clsp\{\xi_y : y \in [x]_\sigma\}$ of
$\ell^2(\Lambda^\infty)$. For $y \in [x]_\sigma$ and $s(\lambda)=r(y)$, we have $\lambda y \in
[x]_\sigma$, and hence $\Hh_x$ is invariant for $S_\lambda$. On the other hand, $S_\lambda^*\xi_y$
vanishes unless $y(0,d(\lambda))=\lambda$, and then $S_\lambda^*\xi_y=\xi_{\sigma^{d(\lambda)}(y)}$, which
also belongs to $\Hh_x$. Thus $\Hh_x$ is reducing for $\pi_S$, and $\phi_x : a \mapsto
\pi_S(a)|_{\Hh_x}$ is a homomorphism of $C^*(\Lambda)$ into $\Bb(\Hh_x)$. Since $\phi_x(s_{r(x)})
\xi_x = \xi_x \not= 0$, $\ker\phi_x$ is not all of $C^*(\Lambda)$. Equation~\eqref{eq:Wx property},
on the other hand, implies that $\phi_x(s_v) \xi_y = 0$ for all $y \in [x]_\sigma$, so
$s_v\in\ker\phi_x$. Thus $C^*(\Lambda)$ is not simple.
\end{proof}

\end{document}